\documentclass[11pt]{article}
\usepackage[a4paper,left=22mm,top=23mm,right=22mm,bottom=30mm]{geometry}

\usepackage[colorlinks=true,
			citecolor =blue,
			linkcolor=blue
]{hyperref}

\usepackage{amsmath,amsopn,amsthm,amssymb}
\usepackage{mathtools}
\usepackage{mathrsfs}  
\usepackage{enumerate}
\usepackage[shortlabels]{enumitem}
\usepackage{graphicx}
\usepackage{graphics}
\usepackage{xspace}

\usepackage{verbatim}
\usepackage{url}

\usepackage{authblk}
\usepackage[numbers,sort&compress]{natbib}

\usepackage[mathcal]{euscript}
\usepackage{mathptmx}

\usepackage[margin=30pt,font=small,labelfont=bf]{caption}

\newtheorem{theorem}{Theorem}[section]
\newtheorem{lemma}[theorem]{Lemma}
\newtheorem{proposition}[theorem]{Proposition}
\newtheorem{corollary}[theorem]{Corollary}

\newtheorem{problem}{Problem}
\newtheorem{obs}[theorem]{Observation}

\newtheorem{definition}[theorem]{Definition}%

\newtheorem{innercustomthm}{Characterization}
\newenvironment{customthm}[1]
  {\renewcommand\theinnercustomthm{#1}\innercustomthm}
  {\endinnercustomthm}

\makeatletter
\def\moverlay{\mathpalette\mov@rlay}
\def\mov@rlay#1#2{\leavevmode\vtop{%
   \baselineskip\z@skip \lineskiplimit-\maxdimen
   \ialign{\hfil$\m@th#1##$\hfil\cr#2\crcr}}}
\newcommand{\charfusion}[3][\mathord]{
    #1{\ifx#1\mathop\vphantom{#2}\fi
        \mathpalette\mov@rlay{#2\cr#3}
      }
    \ifx#1\mathop\expandafter\displaylimits\fi}
\makeatother

\newcommand{\cupdot}{\charfusion[\mathbin]{\cup}{\cdot}}

\usepackage{xcolor}
\definecolor{MHcol}{RGB}{0, 100, 250}

\newcommand{\mh}[1]{\begingroup\color{MHcol}#1\endgroup}

\DeclareMathOperator{\med}{med}
\newcommand{\mc}{{\textrm{medico}}\xspace}
\newcommand{\MC}{{\textrm{Medico}}\xspace}

\providecommand{\keywords}[1]{\textbf{\textit{Keywords: }} #1}

  \title{On a generalization of median graphs: $k$-median graphs}

\author[ ]{Marc Hellmuth*} 
\author[ ]{Sandhya Thekkumpadan Puthiyaveedu}

\affil[ ]{Department of Mathematics, Faculty of Science,
	Stockholm University, SE - 106 91 Stockholm,   Sweden \smallskip %\newline 
	}

\affil[*]{corresponding author 
			(\textnormal{\texttt{marc.hellmuth@math.su.se}})}

\date{\ }

\begin{document}

\maketitle

\abstract{Median graphs are connected graphs in which for all three vertices
          there is a unique vertex that belongs to shortest paths between each
          pair of these three vertices. To be more formal, a graph $G$ is
          a median graph if, for all $\mu, u,v\in V(G)$, it holds that
          $|I(\mu,u)\cap I(\mu,v)\cap I(u,v)|=1$ where $I(x,y)$ denotes the set
          of all vertices that lie on shortest paths connecting $x$ and $y$. In
          this paper we are interested in a natural generalization of median graphs,
          called $k$-median graphs. A graph $G$ is a $k$-median graph, if
          there are $k$ vertices $\mu_1,\dots,\mu_k\in V(G)$ such that, for all
          $u,v\in V(G)$, it holds that $|I(\mu_i,u)\cap I(\mu_i,v)\cap I(u,v)|=1$,
          $1\leq i\leq k$. By definition, every median graph with $n$ vertices
          is an $n$-median graph. We provide several characterizations of
          $k$-median graphs that, in turn, are used to provide many novel characterizations
          of median graphs.}
\smallskip

\noindent 
\keywords{median graph, convexity, meshed and quadrangle property, modular, interval
}

\sloppy

\section{Introduction}

A median graph is a connected graph, in which, for each triple of vertices
$x,y,z$ there exists a unique vertex $\med(x,y,z)$, called the median,
simultaneously lying on shortest paths between each pair of the triple
\cite{Birkhoff:47}. Denoting with $I(u,v)$ the set of all vertices that lie on
shortest paths connecting $u$ and $v$ and putting $I(x,y,z)\coloneqq I(x,y)\cap
I(x,z)\cap I(y,z)$, a graph is a median graph precisely if $|I(x,y,z) | = 1$ for
all of its vertices $x,y,z$ \cite{mulder1980interval,mulder1998metric}. Median
graphs have been studied at least since the 1940's
\cite{Birkhoff:47,Avann:61,Nebesky:71} and naturally arise in several fields of
mathematics, for example, in algebra \cite{bandelt1983median}, metric graph
theory \cite{Bandelt:08}, geometry \cite{chepoi2000graphs} or the study of 
split systems \cite{hellmuth2022injective,Dress1997,buneman71recovery}. Moreover, they
have practical applications in areas such as social choice theory
\cite{Bandelt:84,Day:03}, phylogenetics
\cite{Dress:97,BSH:22,hellmuth2022injective}, and forensic science
\cite{parson2007empop}.
 
In this paper, we study a generalization of median graphs, that is, $k$-median
graphs. A vertex $\mu$ in a graph $G$ is called median-consistent (\mc), if
$|I(\mu,x,y) | = 1$ for all $x,y\in V(G)$. A graph $G$ is a $k$-median graph if
it contains $k$ \mc vertices. In other words, $G$ is a $k$-median graph if the
set $W$ obtained from $V(G)$ by removal of all vertices $x$ for which $|I(x,y,z)
| \neq 1$ satisfies $|W|\geq k$. The existence of $\mc$ vertices has been
	studied in a work of Bandelt et al.\ \cite{BVV:93} in the context of modular
graphs, that is, graphs that satisfy $I(x,y,z) \neq \emptyset$ for all $x,y,z\in
V(G)$. In \cite{BVV:93}, \mc vertices were called ``neutral'' and the authors
provided a simple characterization of \mc vertices in modular graphs $G$ in
terms of intervals $I(x,y)$ and $K_{2,3}$ subgraphs (cf.\ \cite[Prop.\ 5.5]{BVV:93})

One may think of the integer $k$ in $k$-median graphs as a measure on how
``close'' a given graph is to a median graph as every median graph $G$ is a
$|V(G)|$-median graph. Studying $k$-median graphs is also motivated by the
following observations. In \cite{BSH:22}, the authors were interested in
``representing'' edge-colored graphs by rooted median graphs. To be more
precise, given an edge-colored graph $H$ the task is to find a vertex-colored
median graph $G$ with a distinguished vertex $\mu$ (called root in
\cite{BSH:22}) such that the color of the unique median $\med(\mu,x,y)$
corresponds to the color of the (non)edges $\{x,y\}$ in $H$ for all $x,y\in
V(H)$. The latter, in particular, generalizes the concept of so-called cographs
\cite{Corneil:81}, symbolic ultrametrics \cite{Boecker:98,HW:16a,Hellmuth:13a},
or unp 2-structures \cite{EHPR:96,HSW:16}, that is, combinatorial objects that
are represented by rooted vertex-colored trees. The idea of using median graphs
instead of trees can be generalized even more by asking for an arbitrary
vertex-colored graph $G$ that contains a vertex $\mu$ for which the median
$\med(\mu,x,y)$ is well-defined for all $x,y\in V(G)$ and the color of
$\med(\mu,x,y)$ distinguishes between the colors of the underlying (non)edges
$\{x,y\}$ in the edge-colored graph $H$. In this case, it is only required that
$G$ is a 1-median graph and thus, contains (at least) one \mc vertex. 

In this paper, we study the structural properties of $k$-median graphs which, in
turn, lead also to novel characterizations of median graphs. By way of example,
we show that median graphs are precisely those graphs $G$ for which $I(x,y,z)
\neq \emptyset$ and the subgraph induced by the vertices in $I(x,y,z)$ is
connected for all $x,y,z\in V(G)$ (cf.\ Thm.\ \ref{thm:med-novel}). This paper
is organized as follows. In Section \ref{sec:prelim}, we introduce necessary
notation and definitions. We then summarize our main results in Section
\ref{sec:basic-k-median}. In Section \ref{sec:conditions}, we establish basic
structural properties and provide necessary conditions for $k$-median graphs. In
Section \ref{sec:convex} we provide characterizations of $k$-median that, in
turn, result in plenty of new characterizations of median graphs. We complement
these results by a simple program to verify if a given graph is a $k$-median
graph and to compute the largest such integer $k$ in the affirmative case. This
straightforward algorithm is written in Python, hosted at GitHub
\cite{github-MH}) and can be used to verify the examples.

\section{Preliminaries}
\label{sec:prelim}

For a set $A$, we write $A^n \coloneqq \times_{i=1}^n A$ for the $n$-fold
Cartesian set product of $A$. All graphs $G=(V,E)$ considered here are
undirected, simple and finite and have vertex set $V(G)\coloneqq V$ and edge set
$E(G)\coloneqq E$. We put $|G| \coloneqq |V|$ and $\| G \|\coloneqq |E|$. We
write $G[W]$ for the graph that is induced by the vertices in $W\subseteq V$. A
graph $G$ is \emph{$H$-free} if $G$ does not contain an induced subgraph
isomorphic to $H$. For two graphs $G$ and $H$, their \emph{intersection $G\cap H$} is
defined as the graph $(V(G)\cap V(H), E(G)\cap E(H))$. A \emph{complete graph}
$K_{|V|}=(V,E)$ is a graph that satisfies $\{u,v\}\in E$ for all distinct
$u,v\in V$. A graph $G=(V,E)$ is bipartite if its vertex set $V$ can be
partitioned into two non-empty sets $V_1$ and $V_2$ such that $\{x,y\}\in E$
implies $x\in V_i$ and $y\in V_j$, $i\neq j$. A bipartite graph $G=(V,E)$ with
bipartition $V_1\cupdot V_2=V$ is \emph{complete} and denoted by
$K_{|V_1|,|V_2|}$, if $x\in V_i$ and $y\in V_j$, $i\neq j$ implies $\{x,y\}\in
E$. A \emph{cycle $C$} is a connected graph in which all vertices have degree
$2$. A cycle of length $n=|C|$ is denoted by $C_n$. A \emph{hypercube $Q_n$} has
vertex set $\{0,1\}^n$ and vertices are adjacent if their coordinates differ in
precisely one position. We write $Q_n^-$ for the hypercube $Q_n$  
from which one vertex and its incident edges has been deleted.  

All paths in $G=(V,E)$ are considered to be simple, that is, no vertex is traversed
twice. The distance $d_G(u,v)$ between vertices $u$ and $v$ in a graph $G$ is
the length $\|P\|$ of a shortest path $P$ connecting $u$ and $v$. Note that the
distance for all vertices is well-defined whenever $G$ is connected and that
$d_G(x,y)= d_G(y,x)$. For a subset $W\subseteq V$ we say that $w\in W$ is 
(among the vertices in $W$) \emph{closest}
to $v\in V$ if $d_G(v,w)=\min_{x\in W} d_G(v,x)$. 
We call paths connecting $u$ and $v$ also $(u,v)$-paths. A
path %\OLD{of length $n=\|P\|$}
on $n=|P|$ vertices is denoted by $P_n$. A subgraph $H$ of $G$ is
\emph{isometric} if $d_H (u, v) = d_G (u, v)$ for all $u, v \in V (H)$. A
subgraph $H$ of $G$ is \emph{convex} if for any $u, v \in V(H)$, all shortest
$(u, v)$-paths belong to $H$. The \emph{convex hull} of a subgraph $H$ in
$G$ is the least convex subgraph of $G$ that contains $H$. 
If the context is clear, we may write
$d(\cdot,\cdot)$ instead of $d_G(\cdot,\cdot)$ for a given graph $G$.

For later reference, we provide the following 
\begin{lemma}\label{lem:01-diff-adge}
	For every edge $\{u,v\}$ in a connected graph $G$ and every vertex $x\in
	V(G)$ it holds that $0\leq |d_G(x,u)-d_G(x,v)|\leq 1$. Moreover, $G$ is
	bipartite if and only if $d_G(x,u) \neq d_G(x,v)$ for all $x\in V(G)$ and
	$\{u,v\}\in E(G)$. In particular, if there is a vertex $x$ such that
	 $d_G(x,u) \neq d_G(x,v)$ for all	$\{u,v\}\in E(G)$, then $G$ is bipartite. 
\end{lemma}
\begin{proof}
	The first statement is an easy consequence of the triangle inequality, 
	that is, $d_G(x,u)\leq d_G(x,v) + d_G(u,v)$ implies 
	$|d_G(x,u)-d_G(x,v)|\leq d_G(u,v) = 1$ for all edges $\{u,v\}$ and all vertices $x$.
	The second statement is equivalent to \cite[Thm.\ 2.3]{ovchinnikov2011graphs}. 
	Suppose now that there is a vertex $x$ in $G=(V,E)$ such that
	 $d_G(x,u) \neq d_G(x,v)$ for all edges $\{u,v\}\in E$. 
	 Let $V_1\coloneqq \{w\in V\mid d_G(x,w) \text{ is odd}\}$. 
	 $V_2\coloneqq \{w\in V\mid d_G(x,w) \text{ is even}\}$ and  
	 $\{u,v\}\in E$. Since $d_G(x,u) \neq d_G(x,v)$ and by the first 
	 statement, it holds $u\in V_i$ and $v\in V_j$, $i\neq j$.
	 As this holds for all edges in $E$, the sets $V_1,V_2$ form 
	 a valid bipartition of $G$.
\end{proof}

The \emph{interval} between $x$ and $y$ is the set $I_G(x,y)$ of all vertices
that lie on shortest $(x,y)$-paths. A vertex $x$ is a \emph{median} of a triple
of vertices $u$, $v$ and $w$ if $d(u,x)+d(x,v)=d(u,v)$, $d(v,x)+d(x,w)=d(v,w)$
and $d(u,x)+d(x,w)=d(u,w)$. Equivalently, $x$ is a median of $u$, $v$ and $w$ if
$x\in I_G(u,v,w)\coloneqq I_G(u,v)\cap I_G(u,w)\cap I_G(v,w)$
\cite{mulder1980interval}. If $I_G(u,v,w) = \{x\}$ consist of $x$ only, we put
$\med_G(u,v,w)\coloneqq x$ and say that \emph{$\med_G(u,v,w)$ is well-defined}.
A graph $G$ is a \emph{median graph} if, for every triple $u$, $v$ and $w$ of
its vertices, $\med_G(u,v,w)$ is well-defined. \cite{Mulder:78, SM:99}. A graph
is $G$ \emph{modular}, if $I_G(u,v,w)\neq \emptyset$ for all $u,v,w\in V$. One
easily verifies that median graphs are modular and that modular graphs must be connected. 
Following \cite{mulder1980interval}, a graph $G$ is called \emph{interval-monotone} if,
for all $u,v\in V$, $I(x,y) \subseteq I(u,v)$ for all $x,y\in I(u,v)$, i.e., the
induced subgraph $G[I(u,v)]$ is convex. 
If the context is clear, we may write $I(\cdot,\cdot)$, resp., $I(\cdot,\cdot,\cdot)$
instead of $I_G(\cdot,\cdot)$, resp., $I_G(\cdot,\cdot,\cdot)$ for a given graph
$G$.

For later reference, we summarize here the results Prop.\ 1.1.2(iii) and 1.1.3. established
in \cite{mulder1980interval}. 
\begin{lemma}\label{lem:mulder-interval}
Let $G$ be a connected graph. Then, $x\in I_G(u,v)$ implies that $I_G(u,x)\subseteq I_G(u,v)$. 
Moreover, for any three vertices $u,v,w$ of $G$ there exists a vertex $z\in I_G(u,v)\cap I_G(u,w)$
such that $I_G(z,v)\cap I_G(z,w) = \{z\}$.
\end{lemma}

In the upcoming proofs the following definition will play a particular role.
\begin{definition}[distance-$\ell$-static and meshed]
	Let $G=(V,E)$ be a graph.  A quartet $(x,z,y,w)\in V^4$ 
	is \emph{distance-$\ell$-static} for some $\ell\geq 1$, if
	$d_G(x, w) = d_G(z, w) = \ell = d_G(y, w) - 1$ and
	$d_G(x, z) = 2$ with $y$ a common neighbor of $x$ and $z$. 

	$G$ is called \emph{meshed} if all distance-$\ell$-static quartets
	$(x,z,y,w)\in V^4$ satisfy the \emph{quadrangle property}: \emph{there
	exists a common neighbor $u$ of $x$ and $z$ with $d_G(u, w) = \ell - 1$.}
\end{definition}
Meshed graphs play a central role in the characterization of median graphs
\cite{BRESAR2002149,SM:99}. The vertices $x,z,y,w$ in an distance-$\ell$-static
quartet $(x,z,y,w)$ must, by definition, be pairwise distinct. Note that
$(x,z,y,w)$ is distance-$\ell$-static precisely if $(z,x,y,w)$ is
distance-$\ell$-static. The order of the remaining vertices  %\OLD{ Moreover, we
%emphasize that the order of the vertices} 
in an distance-$\ell$-static quartet
$(x,z,y,w)$ matters, since the last vertex $w$ serves as ``reference'' to the
vertex to which the distances are taken for $x,y,z$ and 
the third vertex $y$ implies that  implies that $x,y,z$ must induce a
path $P_3$ with edges $\{x,y\}$ and $\{y,z\}$.

\section{Main results}
\label{sec:basic-k-median}

Consider a graph $G = (V,E)$ and let $W\subseteq V$ denote the set of vertices
obtained from $V$ by removing any three vertices $u,v,w$ for which
$\med_G(u,v,w)$ is not well-defined. In this case, $W$ contains all vertices
$\mu$ for which $\med_G(\mu,v,w)$ is well-defined for all $v,w\in V$ which gives
rise to the following
\begin{definition}[\mc vertices and $k$-median graph]
A vertex $\mu\in V$ is \emph{median-consistent (\mc)} in a graph $G$ if
$\med_G(\mu,v,w)$ is well-defined for all distinct $v,w\in V(G)\setminus
\{\mu\}$. 

A graph is a \emph{$k$-median graph} if it has $k\geq 1$ \mc vertices. Moreover,
a $k$-median graph is a \emph{proper} $k$-median graph if it has precisely $k$
\mc vertices.
\end{definition}

\begin{figure}[t] 
\centering
\includegraphics[width=.7\textwidth]{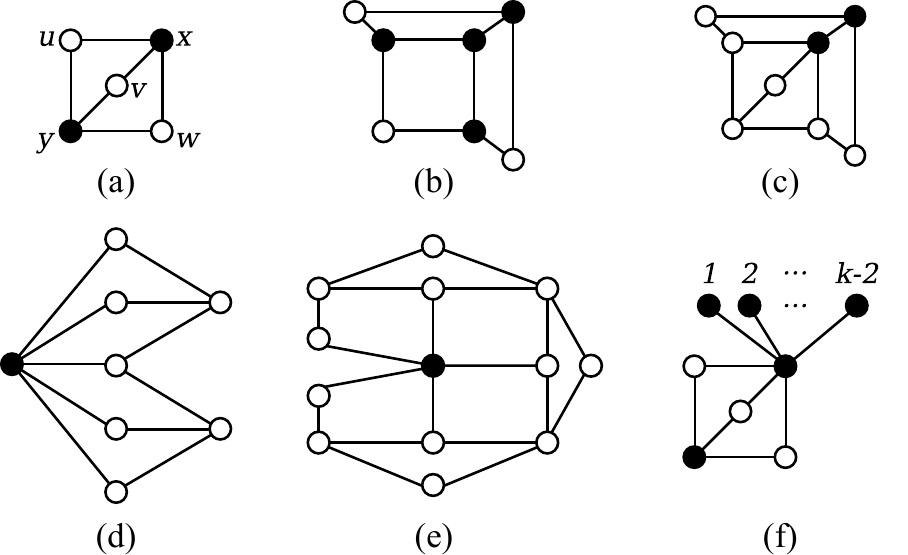}
\caption{Shown are several $k$-median graphs. The respective \mc vertices are
         highlighted as black vertices. Shown are (a) a $K_{2,3}$, (b) a
         $Q_3^-$, (c) a combination of both a $K_{2,3}$ and $Q_3^-$, (d) a
         $Q_3^-$-free proper 1-median graph, (e) a
         $K_{2,3}$-free proper 1-median graph. (f) This graph serves as a generic example of a
         proper $k$-median graph for all $k\geq 2$.}
\label{fig:exmpls}
\end{figure}

Note that if a graph $G$ is not connected, then $I_G(x,y,z)=\emptyset$ for any
two vertices $x,y$ in one connected component and $z$ in another connected
component. Hence, disconnected graphs cannot contain any \mc vertex. We
summarize this finding and other simple results in 
\begin{obs} \label{obs:summ}
Let $G$ be a graph. 
\begin{itemize}
\item If $G$ contains a \mc vertex $\mu$, then $G$ is connected. 
\item For every $k\geq 1$, there is a proper $k$-median graph (cf.\ Fig.\
      \ref{fig:exmpls}(d,e,f)).
\item Every $k$-median graph is an $\ell$-median graph for all $\ell\in
      \{1,\dots,k\}$.
\item  A graph $G=(V,E)$ is a \emph{median graph} if and only if $G$ is
       $|V|$-median graph. 
\end{itemize}
\end{obs}

The following result shows that there are no proper $k$-median graphs $G$ with
$k\geq |V|-2$ and provides a new although rather simple characterization of
median graphs.
\begin{proposition}\label{prop:n-12IMPLIESmedianG}
A graph $G=(V,E)$ is a median graph if and only if $G$ is a $(|V|-1)$- or a
$(|V|-2)$-median graph. Hence, there is no proper $(|V|-1)$- or $(|V|-2)$-median
graph. 
\end{proposition}
\begin{proof}
If $G=(V,E)$ is median graph, then the last two statements in Obs.\
\ref{obs:summ} imply that $G$ is a $(|V|-1)$- or $(|V|-2)$-median graph. By
contraposition, assume that $G$ is not a median graph. In this case, there is at
least one vertex $x$ that is not a \mc vertex of $G$. Thus, there are two
vertices $y$ and $z$ such that $x,y$ and $z$ are pairwise distinct and
$\med(x,y,z)$ is not well-defined. Hence, neither $y$ nor $z$ can be a \mc
vertex. Hence, $G$ is either no $k$-median graph at all or, if $G$
is a $k$-median graph, then
 $k\leq|V|-3$ which implies that $k\notin \{|V|-1, |V|-2\}$.
\end{proof}

In contrast to median graphs, $k$-median graphs may contain an induced $K_{2,3}$
(cf.\ Fig.\ \ref{fig:exmpls}). Moreover, median graphs are characterized as
those graphs $G$ for which the convex hull of every isometric cycle $C$ is a hypercube (cf.\
\cite[Thm.\ 5]{SM:99}). As a $Q_3^-$ is a $4$-median graph that contains an
isometric cycle $C_6$, the latter property is, in general, not satisfied for
$k$-median graphs.

We give now a brief overview of our main results. In Section
\ref{sec:conditions}, we provide several necessary conditions and show, among
other results, that every $k$-median is bipartite, $K_{3,3}$-free and that every
edge $e\in E(C)$ of every cycle $C$ of $G$ is also part of an induced $C_4$ in
$G$. Moreover, we characterize the existence of induced $K_{2,3}$s 
in $k$-median graphs $G$ and 
show, in addition, that $G$ contains a $Q_3^-$ precisely if it contains an induced $C_6$.
We then continue in Section \ref{sec:convex} with a generalization of
convex subgraphs. 

\begin{definition}[$v$-convex]
	Let $G$ be a  graph and $v\in V(G)$. A subgraph $H$ of $G$ is
	\emph{$v$-convex} if $v\in V(H)$ and every shortest path connecting $v$ and
	$x$ in $G$ is also contained in $H$, for all $x\in V(H)$.
\end{definition} 
Note that $v$-convex subgraphs $H$ of $G$ satisfy $I_G(v,x)\subseteq V(H)$ for
all $x\in V(H)$. Moreover, $v$-convex subgraphs are connected but not
necessarily induced, isometric or convex, see Fig.\ \ref{fig:exmpl-v-convex} for
an example. We then derive

\begin{customthm}{1}[Thm.\ \ref{thm:char-k-median}]
	$G=(V,E)$ is a $k$-median graph if and only if there $k$ vertices
	$\mu_1,\dots,\mu_k\in V$ such that $\mu_i$ is a \mc vertex in every
	$\mu_i$-convex subgraph of $G$, $1\leq i\leq k$.
\end{customthm}

\begin{figure}[t]
\centering
\includegraphics[width=.8\textwidth]{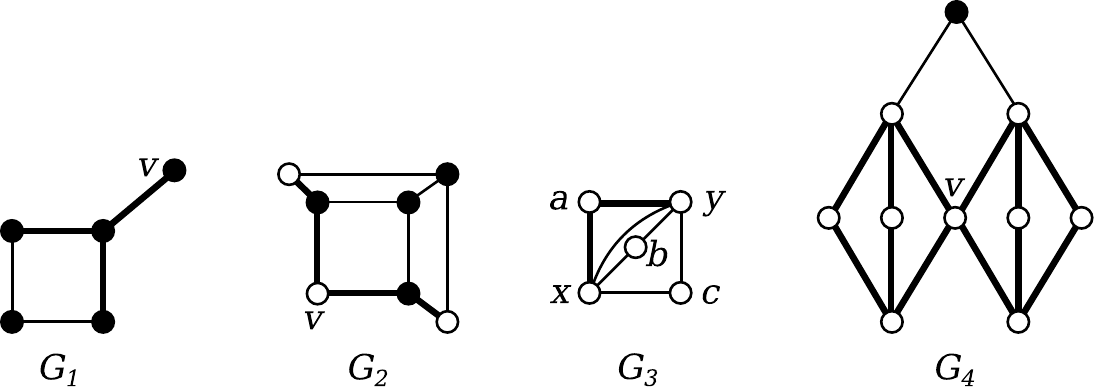}
\caption{Shown are four graphs $G_1,G_2,G_3$ and $G_4$. \MC vertices are
         highlighted as black vertices and subgraphs $H_i$ of $G_i$, $1\leq i\leq 4$,
         are highlighted by thick edges. 
         All $H_i$ are $v$-convex subgraphs of $G_i$ but not convex.
    	  Since $G_1$ is a
         median graph and $v$ a \mc vertex of $G_1$, $H_1$ is isometric and thus,
         induced (cf.\ Lemma \ref{lem:rho-convex=>isometric}). In contrast, $v$
         is not a \mc vertex in the graph $G_2$. Although $H_2$ is $v$-convex and
         induced, it is not isometric. The graph $G_3$ does not contain any \mc
         vertices. The subgraph $H_3$ of $G_3$ is $a$-convex, but not induced.
         Adding the edge $\{x,y\}$ to $H_3$ would yield another $a$-convex
         subgraph that is also convex. 
         Note that $I\coloneqq I_{G_3}(a,b,c) = \{x,y\}$. Hence,
         $G_3[I]\simeq K_2$ consists of the edge $\{x,y\}$ and is thus,
         connected and non-empty. 
         The $v$-convex subgraph $H_4$ of $G_4$ is not a $1$-median graph.
				         According to Lemma \ref{lem:v-convex-bip}, 
          $v$-convex subgraphs with specified vertex set
         are uniquely determined in bipartite graphs.
}
\label{fig:exmpl-v-convex}
\end{figure}

This, in turn, yields a well-known characterization of median graphs stating
that $G$ is a median graph if and only if every convex subgraph of $G$ is a
median graph, see e.g.\ \cite{mulder1980interval}.
We then make frequent use of the following conditions (C0), (C1) 
and (C2).

\begin{definition} 
Let  $G=(V,E)$ be a graph and $u\in V$. Then, $G$ satisfies 	
\begin{description}%[noitemsep]
	
	\item[$\quad$ \textbf{(C0) w.r.t.\ $\mathbf{u}$} \textnormal{\ if:}] 
			$I_G(u,v)\cap I_G(v,w) = \{v\}$ implies $v\in I_G(u,w)$, for all
			$v,w\in V$. 

	\item[$\quad$ \textbf{(C1) w.r.t.\ $\mathbf{u}$} \textnormal{\ if:}] 
		$G[I_G(u,v,w)]$ contains at least one vertex and is connected, for all
		$v,w\in V$.

	\item[$\quad$ \textbf{(C2) w.r.t.\ $\mathbf{u}$} \textnormal{\ if:}] 
				$G[I_G(u,v,w)]$ contains an edge whenever it contains more than one vertex, 

\hspace{2.cm}				 for all		$v,w\in V$.
\end{description}
\end{definition}
Note that $G$ satisfies (C0) w.r.t.\ $u$ if and only if $I_G(u,v)\cap I_G(v,w) =
\{v\}$ implies that $d_G(u,v)+d_G(v,w) = d_G(u,w)$ for all $v,w\in V$. Moreover,
if $G$ satisfies (C1) w.r.t.\ some $u\in V$, then $G$ must be connected as,
otherwise, $I_G(u,v,w)=\emptyset$ for $u$ and $v$ being in distinct connected
components. One easily verifies that every median graph satisfies (C0), (C1) and
(C2) w.r.t.\ each of its vertices. A simple example of a graph that does neither
satisfy (C0) nor (C1) w.r.t.\ some vertex $u$ is an induced $C_6$. To see this,
let $G\simeq C_6$ and $v$ and $w$ be the vertices in $G$ that have distance 2 to
$u$. In this case, $G[I(u,v,w)]$ does not contain any vertices and $I_G(u,v)\cap
I_G(v,w) = \{v\}$ although the unique shortest path connecting $u$ and $w$ does
not contain $v$, that is, $v\notin I_G(u,w)$. In contrast, $G\simeq C_6$
trivially satisfies (C2) w.r.t.\ all of its vertices, since $|I(u,v,w)|$ is
either empty or contains precisely one element for all $u,v,w\in V$ Consider now
the graph $G\coloneqq G_3$ in Fig \ref{fig:exmpl-v-convex}. The set $I_G(a,b,c)
= \{x,y\}$ is the only one among the other $I_G(a,\cdot,\cdot)$ that contains
more than one vertex. Moreover, $G[I_G(a,b,c)]$ contains the edge $\{x,y\}$ and
is, thus, connected. Hence $G$ satisfies (C2) w.r.t.\ $a$. However,
$I_G(a,x,y)=\emptyset$ which implies that $G$ does not satisfy (C1) w.r.t.\ $a$.
Moreover, $G$ does not satisfy (C0) w.r.t.\ $a$, since $I_G(a,x)\cap
I_G(a,y)=\{a\}$ but $a\notin I_G(x,y)$.

We provide first a characterization that is based on Condition (C1). 
\begin{customthm}{2}[Thm.\ \ref{thm:char3}]
A graph $G$ is a $k$-median graph if and only if $G$ satisfies (C1) w.r.t.\
w.r.t.\ $\mu_1,\dots,\mu_k\in V(G)$. In this case, the vertices
$\mu_1,\dots,\mu_k\in V$ are \mc vertices of $G$.
\end{customthm}

Mulder provided the following characterization of median graphs.
\begin{theorem}[{\cite[Thm.\ 3.1.7]{mulder1980interval}}]\label{thm:mulder-median}
	A graph $G$ is a median graph if and only if 
	$G$ is connected, interval-monotone and satisfies (C0) w.r.t.\ all of its vertices.  
\end{theorem}
By definition, $k$-median graphs are connected and, as shall we see later in Lemma \ref{lem:C0}, 
always satisfy (C0) w.r.t.\ $k$ of its vertices. However, $k$-median
graphs are, in general, not interval-monotone. 
The simplest example is possibly the $2$-median graph $K_{2,3}$ as
shown in Fig.\ \ref{fig:exmpls}(a) with bipartition $\{x,y\}\cupdot \{u,v,w\}$. 
One observes that  $I(u,v)=\{u,v,x,y\}$ and, since $w\in I(x,y)$,
it holds that $I(x,y)\not\subseteq I(u,v)$. Hence, the subgraph 
induced by $I(u,v)$ is not convex. This
begs the question to what extent Theorem \ref{thm:mulder-median}
can be generalized to cover the properties of $k$-median graphs. 
As it turns out, interval-monotonicity can be 
replaced by Condition (C2) together with bipartiteness.

\begin{customthm}{3}[Thm.\  \ref{thm:char-C0}]
A graph $G$ is a $k$-median graph if and only if $G$ is connected, bipartite and
satisfies (C0) and (C2) w.r.t.\ $\mu_1,\dots,\mu_k\in V(G)$. In this case, the
vertices $\mu_1,\dots,\mu_k\in V$ are \mc vertices of $G$.
\end{customthm}

The latter results naturally translate into 
novel characterizations of median graphs. 

\begin{customthm}{4}[Thm.\ \ref{thm:char-k-median}, \ref{thm:med-novel} \& \ref{thm:imrich}]
For every graph $G$, the following statements are equivalent.
\begin{enumerate}
\item  $G$ is a median graph.
\item Every $v\in V(G)$ is a \mc vertex in every $v$-convex subgraph of $G$.
\item $G$ satisfies (C1) w.r.t.\ all of its vertices.
\item $G$ is connected, bipartite and satisfies (C0) and (C2) w.r.t.\ all of its vertices.
\item The graph $G(u,v,w)\coloneqq G(u,v)\cap G(u,w)\cap G(v,w)$ is not empty
      and connected for all $u,v,w\in V$ where $G(x,y)$ denotes the subgraph of
      $G$ with $V(G(x,y)) = I_G(x,y)$ and where $E(G(x,y))$ consists precisely
      of all edges that lie on the shortest $(x,y)$-paths, $x,y\in V$.
\end{enumerate}
\end{customthm}

\section{Necessary Conditions and Subgraphs}
\label{sec:conditions}

We provide in this section several properties that must be satisfied by every
$k$-median graph. We start with considering distances of \mc vertices to
adjacent vertices and to vertices on isometric cycles.

\begin{lemma}\label{lem:adjdisttomedian}
	Let $G=(V,E)$ be a $k$-median graph and $\mu$ be a \mc vertex in $G$. Then
	$|d_G(\mu,u)-d_G(\mu,v)|=1$ for all edges
	$\{u,v\}\in E$. 
\end{lemma}
\begin{proof}
	Let $G=(V,E)$ be a $k$-median graph with \mc vertex $\mu$ and $\{u,v\}\in E$
	be an edge. If $u=\mu$ or $v=\mu$, then the statement is vacuously true.
	Hence, suppose that $u,v\neq \mu$. Since $\{u,v\}\in E$, we
	have $I_G(u,v) =\{u,v\}$. Moreover, since $\mu$ is a \mc vertex in $G$ it
	must hold $|I_G(\mu,u)\cap I_G(\mu,v)\cap I_G(u,v)| =1$. Hence, we may
	assume, w.l.o.g., that $\med(\mu,u,v)=u$. Then, by the definition of
	medians, $d_G(\mu,v)=d_G(\mu,u)+d_G(u,v)=d_G(\mu,u)+1$ and, therefore,
	$d_G(\mu,v) - d_G(\mu,u)=1$.  
\end{proof}

Lemma \ref{lem:01-diff-adge} and \ref{lem:adjdisttomedian} imply
\begin{proposition}\label{prop:bipthm}	
 Every $k$-median graph is bipartite.
\end{proposition}

\begin{lemma}\label{lem:cndist} 
Let $G=(V,E)$ be a $k$-median graph and $\mu$ be a \mc vertex in $G$. Moreover,
suppose that $G$ contains an induced cycle $C$. Let $v\in V(C)$ be a vertex in
$C$ that is closest to $\mu$ with $K\coloneqq d_G(v,\mu) = \min_{w\in V(C)}
d_G(w,\mu)$. Furthermore, let $u\in V(C)$ and put $i\coloneqq d_C(u,v)$. Then, 
	\begin{equation*}
		d_G(\mu,u)\in
		\begin{cases}
			\{K+1,K+3,\ldots , K+i\} & \text{if $i$ is odd}\\
			\{K, K+2,\ldots , K+i\} & \text{if $i$ is even.}
			\end{cases}       
	\end{equation*}
	
	\end{lemma}
\begin{proof} 
Let $G=(V,E)$ be a $k$-median graph, $\mu$ be a \mc vertex in $G$ and $C$ be an
induced cycle in $G$. Moreover, let $v\in V(C)$ be a vertex in $C$ that is
closest to $\mu$ and put $K\coloneqq d_G(v,\mu)$. Now consider a vertex $u$ for
which $d_C(u,v)=i\geq 0$. Note that $i\leq \frac{|V(C)|}{2}$. 

We proceed now by induction on the length $d_C(u,v)=i$. If $i=0$, we have $u=v$
and thus, trivially, $d_G(u,\mu)=K$ and $i$ is even. If $i=1$, then $\{u,v\}\in
E$. Since $v$ is a vertex in $C$ that is closest to $\mu$, we have
$d_G(\mu,u)\geq d_G(\mu,v) = K$. This together with Lemma
\ref{lem:adjdisttomedian} implies that $d_G(\mu,u)-d_G(\mu,v)=1$ and thus,
$d_G(\mu,u)\in \{K+1\}$ and $i$ is odd.

Assume, now that the statement is true for all $i$ with $0\leq i< \frac{|V(C)|}{2}$.
 Let $u\in V(C)$ be a vertex with
$d_C(u,v)=i+1$. As we already verified the cases $d_C(u,v)\in \{0,1\}$, we may
assume that $d_C(u,v)\geq 2$. By Prop.\ \ref{prop:bipthm}, 
$G$ is bipartite and thus, $|C|\geq 4$. Hence, such a vertex $u\in V(C)$ with
$d_C(u,v)\geq 2$ indeed exists. Since $d_C(u,v)\geq 2$, there is a vertex $u'\in
V(C)$ that satisfies $d_C(v,u')=i$ and $\{u',u\}\in E$. If $d_C(u,v)$ is odd,
$d_C(u',v)$ must be even and we obtain, by induction hypothesis, $d_G(\mu,u')\in
\{K, K+2,\ldots , K+i\}$. By Lemma \ref{lem:adjdisttomedian},
$|d_G(\mu,u)-d_G(\mu,u')|=1$. Note that the choice of $v$ and $d_G(v,\mu)=K$
implies that $d_G(\mu,u)=K-1$ is not possible. The latter three arguments now 
imply that $d_G(\mu,u)\in \{K+1, K+3,\ldots ,
K+(i+1)\}$. By similar arguments, if $d_C(u,v)$ is even, $d_C(u',v)$ must be odd
and $d_G(\mu,u')\in \{K+1, K+3,\ldots , K+i\}$ which together with Lemma
\ref{lem:adjdisttomedian} implies that $d_G(\mu,u)\in \{K, K+2, K+4,\ldots ,
K+(i+1)\}$. 
\end{proof}

We provide now a mild generalization of Prop.\ 5.5 in \cite{BVV:93} which states
that every \emph{modular} $k$-median graph must be $K_{3,3}$-free. 

\begin{lemma}\label{lem:K33-free}
Every $k$-median graph is $K_{3,3}$-free. 
\end{lemma}
\begin{proof}
 Let $G=(V,E)$ be a $k$-median graph and $\mu$ be a \mc vertex in $G$. Assume,
 for contradiction, $G$ contains an induced $K_{3,3}\subseteq G$ whose vertex
 set is partitioned into $X =\{x_1,x_2,x_3\}$ and $Y=\{y_1,y_2,y_3\}$ and all
 edges of this $K_{3,3}$ are of the form $\{x_i,y_j\}$, $1\leq i,j\leq 3$. 
	
 Let $v\in X\cup Y$ be one of the vertices that is closest to $\mu$ in $G$,
 i.e., $d_G(\mu,v) = \min_{w\in X\cup Y} d_G(\mu,w)=K\geq 0$. Since the
 $K_{3,3}$ is induced, we can assume w.l.o.g.\ that $v = x_1$. Note that
 $x_1,x_2,y_1,y_2$ induce a $C_4$. 
 Now, we can apply Lemma \ref{lem:cndist} and conclude that $d_G(\mu,
 y_{1})=d_G(\mu, y_{2}) = K+1$ and $d_G(\mu,x_2) \in \{K,K+2\}$.	
	
 Assume, for contradiction, that $d_G(\mu,x_2) = K$. Note that $x_1,x_2\in
 I_G(y_1,y_2)$ since the $K_{3,3}$ is induced. Moreover, one can extend any
 shortest path from $\mu$ to $x_1$ (which has length $K$) by adding the edge
 $\{x_1,y_1\}$ to obtain a shortest path from $\mu$ to $y_1$ of length $K+1$ and
 thus, $x_1\in I_G(\mu,y_1)$. By similar arguments, we obtain $x_2\in
 I_G(\mu,y_1)$ as well as $x_1,x_2\in I_G(\mu,y_2)$. In summary, $x_1,x_2\in
 I_G(\mu,y_1)\cap I_G(\mu,y_2) \cap I_G(y_1,y_2)$. Hence, $\med(\mu,y_1,y_2)$ is
 not well-defined; a contradiction. 

 Thus, $d_G(\mu,x_2) = K+2$ must hold and, by analogous arguments, $d_G(\mu,x_3)
 = K+2$. By similar arguments as before and since $d_G(\mu, y_{1})=d_G(\mu,
 y_{2}) = K+1$, we can conclude that $y_1,y_2\in I_G(\mu,x_2)\cap I_G(\mu,x_3)
 \cap I_G(x_2,x_3)$. Hence, $\med(\mu,x_2,x_3)$ is not well-defined; a
 contradiction. 
\end{proof}

All median graphs are bipartite and meshed \cite{SM:99}. Bipartite meshed
graphs are precisely the modular graphs \cite[Lemma 2.8]{CCHO:20}. However,
since the three non-\mc vertices $u,v,w$ in a $Q_3^-$ satisfy
$I_G(u,v,w)=\emptyset$, it follows that a $Q_3^-$ is not modular and therefore,
not meshed. Moreover, since a $Q_3^-$ is a $4$-median graph, there are
$k$-median graphs that are not meshed. In the upcoming proofs it will,
therefore, be convenient to use a generalization of meshed graphs defined as
follows. 
\begin{definition}[$\mu$-meshed]
A graph $G$ is \emph{$\mu$-meshed} if 
$\mu \in V(G)$ and all distance-$\ell$-static quartets $(x,z,y,\mu)\in V^4$
satisfy the quadrangle property.
\end{definition}

As we shall see next, 
distance-$\ell$-static quadruples $(x,z,y,\mu)$ of $G$ can be
characterized in terms of induced $C_4$s containing $x,y$ and $z$ provided that
$\mu$ is a \mc vertex in $G$. As a consequence, $k$-median graphs are
$\mu$-meshed for all \mc vertices $\mu$. This, in turn, can be used
to show that every edge of a cycle in a $k$-median graphs
is also part of an induced $C_4$. 
The following lemma provides a characterization of distance-$\ell$-static quartets
in $k$-median graphs. Interestingly, this results shows that, in every
$C_4$ of a $k$-median graph, two opposite vertices have both distance
$\ell_{\mu}+1$ while the other two opposite vertices have distance $\ell_{\mu}$, resp., $\ell_{\mu}+2$
to any \mc vertex $\mu$ where $\ell_{\mu}$ is the smallest distance to $\mu$
among the vertices in this  $C_4$.

\begin{lemma}\label{lem:c4claim}
 Let $G$ be a $k$-median graph and $\mu$ be some \mc vertex in $G$. Then, the
 following two statements are equivalent. 
 \begin{enumerate}[noitemsep]
 	\item The quartet $(x,z,y,\mu)\in V^4$ is distance-$\ell$-static. 
 	\item There is an induced $C_4$ in $G$ with vertices $x,z,y,u$ and edges
 	      $\{x,y\}, \{y,z\}, \{z,u\}, \{u,x\}$ where $u$ is a vertex that is
 	      closest to $\mu$ among the vertices $x,y,z,u$.
 \end{enumerate} 
 In particular, every $k$-median graph $G$ is \emph{$\mu$-meshed} for every \mc
 vertex $\mu$ of $G$.
\end{lemma}
\begin{proof} 
 Suppose that $\mu$ is a \mc vertex in the $k$-median graph $G$. Assume first
 that the quartet $(x,z,y,\mu)\in V^4$ is distance-$\ell$-static. Hence, for
 some $\ell\geq 1$, we have $d_G(x, \mu) = d_G(z, \mu) = \ell = d_G(y, \mu) - 1$
 and $d_G(x, z) = 2$ with $y$ a common neighbor of $x$ and $z$. This, in
 particular, implies that $x,y,z,\mu$ are pairwise distinct and
 $\{x,y\},\{y,z\}\subseteq E(G)$ and $\{x,z\}\notin E(G)$. Since $\mu$ is a \mc
 vertex in $G$, $\med_G(\mu,x,y)$ and $\med_G(\mu,x,z)$ are well-defined. Put
 $u\coloneqq \med_G(\mu,x,z)$. We claim that $u\neq x$ and $u\neq z$. Assume,
 for contradiction, that $u=x$. Hence, $x$ lies on a shortest between $\mu$ and
 $z$ and, in particular, $d_G(\mu,z)=d_G(\mu,x)+d_G(x,z) > \ell$; a
 contradiction. Hence, $u\neq x$. By similar arguments, $u\neq z$ must hold.
 Since $d_G(x, z) = 2$ and $u = \med_G(\mu,x,z)$, we have by definition of
 medians that $d_G(x,z)=d_G(x,u)+d_G(u,z)=2$. This together with $u\neq x,z$
 implies that $\{u,x\},\{u,z\}\in E(G)$. Furthermore, $u = \med_G(\mu,x,z)$
 implies $d_G(\mu, x)=d_G(\mu, u)+d_G(u,x)$. Since $d_G(\mu, x)=\ell$ and
 $d_G(u,x)=1$ we obtain $d_G(\mu,u)=\ell-1$. Since $d_G(\mu,u)=\ell-1$ and
 $d_G(y,\mu)=\ell+1$ we have $u\neq y$. Hence, $(x,z,y,\mu)$ satisfies the
 quadrangle property. As the latter arguments hold for any
 distance-$\ell$-static quadruple $(x,z,y,\mu)$, $G$ is $\mu$-meshed for all \mc
 vertices $\mu$ in $G$. Moreover, Lemma \ref{lem:01-diff-adge} implies that $u$
 and $y$ cannot be adjacent. Now, $\{x,y\},\{y,z\}, \{x,u\},\{z,u\}\in E(G)$ but
 $\{x,z\}, \{u,y\}\notin E(G)$ implies that $x,y,z,u$ induce a $C_4$.

 Assume now that there is an induced $C_4$ in $G$ with vertices $x,z,y,u$ and
 edges $\{x,y\}, \{y,z\}, \{z,u\}, \{u,x\}$ and suppose, that $u$ is closest to
 $\mu $ in $G$, i.e., $d_G(\mu,u) = \min_{v\in \{x,y,z,u\}} d_G(v,\mu)$. Suppose
 that $d_G(\mu,u)=\ell-1$ for some $\ell\geq 1$. This and the fact that
 $\{z,u\}$ and $\{u,x\}$ are edges in $G$ implies together with Lemma
 \ref{lem:adjdisttomedian} that $d_G(z,\mu), d_G(x,\mu)\in \{\ell,\ell-2\}$.
 Since, however, $d_G(\mu,u)=\ell-1$ and $u$ is closest to $\mu$ it must hold
 that $d_G(z,\mu) = d_G(x,\mu) =\ell$. Furthermore, $\{x,y\}\in E$ and
 $d_G(x,\mu) =\ell$ together with Lemma \ref{lem:adjdisttomedian} implies that
 $d_G(y,\mu)\in \{\ell-1,\ell+1\}$. Assume, for contradiction, that $d_G(y,\mu)
 =\ell-1$. In this case, $d_G(x,\mu) =\ell$ and $\{u,x\}\in E$ implies that
 there is a shortest $(\mu,x)$-path containing $u$ and thus, $u\in I_G(\mu,x)$.
 By similar arguments, $y\in I_G(\mu,x)$ as well as $u,y\in I_G(\mu,z)$ must
 hold. Moreover, since the $C_4$ in $G$ is induced by the vertices $x,z,y,u$, we
 have $d_G(x,z)=2$ and, by definition of the edges in this $C_4$, we have
 $u,y\in I_G(x,z)$. But then $u,y \in I_G(\mu,x)\cap I_G(\mu,z)\cap I_G(x,z)$,
 which implies that $\med_G(\mu,x,z)$ is not well-defined; a contradiction to
 $\mu$ being a \mc vertex in $G$. Hence, $d_G(y,\mu) =\ell+1$ must hold. It is
 now easy to verify that $(x,z,y,\mu)\in V^4$ is distance-$\ell$-static. 
\end{proof}

\begin{lemma}
Let $G$ be a $k$-median graph, $\mu$ be a \mc vertex and $C$ be a cycle  in $G$. 
Then, there are edges $\{x,y\}, \{y,z\}\in E(C)$ such that 
$(x,z,y,\mu)$ is distance-$\ell$-static for some $\ell\geq 1$. 
\label{lem:C-lstat}
\end{lemma}
\begin{proof}
 Assume, for contradiction, that $(x,z,y,\mu)$ is not distance-$\ell$-static for
 all $\{x,y\}, \{y,z\}\in E(C)$ and any $\ell\geq 1$. Note that $G$ bipartite,
 so $d_G(x,z)=2$ for all $\{x,y\}, \{y,z\}\in E(C)$. Let the vertices of
 $C$ be labeled such that $y_1,\dots,y_{|C|}$ and $\{y_1,y_2\},
 \{y_2,y_3\},\dots, \{y_{|C|},y_1\}\in E(C)$. By Lemma
 \ref{lem:adjdisttomedian}, $|d_G(\mu,y_1)-d_G(\mu,y_2)|=1$ and we can assume
 w.l.o.g.\ that $d_G(\mu,y_1)=k$ and $d_G(\mu,y_2)=k+1$ for some $k\geq 0$.
 Again, by Lemma \ref{lem:adjdisttomedian}, $|d_G(\mu,y_2)-d_G(\mu,y_3)|=1$ and
 so $d_G(\mu,y_2)\in \{k,k+2\}$. Since there is no distance-$\ell$-static
 quadruple using three consecutive vertices in $C$, $d_G(\mu,y_3) = k+2$ must
 hold. Repeating the latter arguments, yields $d_G(\mu,y_{|C|}) = k+|C|-1$.
 Since $\{y_1,y_{|C|}\}\in E(G)$ and by Lemma \ref{lem:adjdisttomedian}, $1 =
 d_G(\mu,y_{|C|})-d_G(\mu,y_1)=k+|C|-1 - k = |C|-1$ which is only possible if
 $|C|=2$; a contradiction. 
\end{proof}

\begin{corollary}
A $k$-median graph cannot contain convex cycles $C_n$, $n>4$.
\end{corollary}
\begin{proof}
Let $C$ be a cycle in $G$ with $n>4$ vertices. Lemma \ref{lem:C-lstat}
implies that there are edges $\{x,y\}, \{y,z\}\in E(C)$ such that 
$(x,z,y,\mu)$ is distance-$\ell$-static for some $\ell\geq 1$. 
By Lemma \ref{lem:c4claim}, there is an induced $C_4$ in $G$ with edges
$\{z,u\}, \{u,x\}$. Hence, there are two shortest $(x,z)$-paths. 
Since $|C|>4$, only one them is located on $C$. Thus, $C$ is not a convex subgraph of $G$. 
\end{proof}

\begin{proposition}\label{prop:cycle-C4}
 Let $G$ be a $k$-median graph. Then, every edge $e\in E(C)$ of every cycle $C$
 of $G$ is an edge of an induced $C_4$. 
\end{proposition}
\begin{proof}
 Let $G=(V,E)$ be a $k$-median graph and $\mu$ be a \mc vertex in $G$. Assume
 that $e = \{x,y\}$ is an edge of some cycle in $G$. Among all cycles that
 contain $e$, let $C$ be one of the cycles of smallest length. Hence $C$ must be
 induced. Assume for contradiction, that $|C|>4$. Thus, $|C|\geq 6$ since $G$ is
 bipartite. By Lemma \ref{lem:adjdisttomedian}, $|d_G(\mu,x)-d_G(\mu,y)|=1$ and
 we may assume w.l.o.g.\ that $d_G(\mu,x) < d_G(\mu,y)$. Let $\{y,z\}\in E(C)$
 be the edge in $C$ that is incident to $y$ and where $z\neq x$. If
 $(x,z,y,\mu)$ is distance-$\ell$-static, then by Lemma \ref{lem:c4claim}, the
 vertices $x,y,z$ and thus, the edge $\{x,y\}$ is contained in a common induced
 $C_4$. Hence, suppose that $(x,z,y,\mu)$ is not distance-$\ell$-static. This
 toget-her with $x\neq z$, $d_G(\mu,x) < d_G(\mu,y)$ and Lemma
 \ref{lem:adjdisttomedian} implies that $d_G(\mu,x) < d_G(\mu,y) < d_G(\mu,z)$.
 Now, label the vertices in $C$ with $v_0, \dots, v_{|C|-1}$ such that
 $v_0\coloneqq x$, $v_1\coloneqq y$, $v_2\coloneqq z$ and $E(C)= \{\{v_0,v_1\},
 \{v_1,v_2\},\dots, \{v_{|C|-1},v_0\} \}$. In the following, we denote with
 $V_j\coloneqq \{v_0,\dots,v_j\}$ the set of vertices of $C$ up to the index
 $j$. By Lemma \ref{lem:C-lstat}, $C$ contains a distance-$\ell$-static quartet
 $(a,c,b,\mu)$, i.e., in particular, $d_G(\mu,a) < d_G(\mu,b)$ and $d_G(\mu,b) >
 d_G(\mu,c)$. Hence, we can choose the smallest index $i\geq 1$ for which
 $d_G(\mu,v_i)=\ell < d_G(\mu,v_{i+1})$ and $d_G(\mu,v_{i+1}) >
 d_G(\mu,v_{i+2})$ is satisfied. Since $d_G(\mu,x) < d_G(\mu,y)$ and by choice
 of $i$, we have $d_G(\mu,v_{j})<d_G(\mu,v_{j+1})$ for all $j\in\{0,\dots,i\}$.
 Lemma \ref{lem:adjdisttomedian} implies that
 $d_G(\mu,v_i)=d_G(\mu,v_{i+2})=\ell$ and $d_G(\mu,v_{i+1})=\ell+1$. Since $v_i$
 and $v_{i+2}$ are distinct vertices, we have $\ell \geq 1$. Moreover, since $G$
 is bipartite, $d_G(v_{i},v_{i+2})=2$ must hold. In summary,
 $(v_{i},v_{i+2},v_{i+1},\mu)$ is distance-$\ell$-static. By Lemma
 \ref{lem:c4claim}, there exists an induced $C_4$ containing the edges $\{v_{i},
 v_{i+1}\}$ and $\{v_{i+1},v_{i+2}\}$ and an additional vertex $u$ such that
 $d_G(\mu,u) = \ell-1$ and, in particular, $\{u,v_{i}\}, \{u,v_{i+2}\}\in E(G)$.
 Note that if $u\in V(C)$, then $C$ is either not induced or $i=1$ and $u=v_0$
 must holds in which case $C\simeq C_4$. Hence, $u\notin V(C)$ and, in
 particular, $u\notin V_i$. By choice of $i$,
 $d_G(\mu,v_{i-1})<d_G(\mu,v_{i})=\ell$ which together with Lemma
 \ref{lem:adjdisttomedian} implies that $d_G(\mu,v_{i-1})= \ell-1$. Moreover,
 since $G$ is bipartite and $\{u,v_{i}\}, \{v_i,v_{i-1}\}\in E(G)$, we have
 $d_G(v_{i-1},u)=2$, Note that $\ell-1 \geq 1$ as otherwise, $d_G(\mu,u) =
 d_G(\mu,v_{i-1})= \ell-1 = 0$ would imply that $u=\mu=v_{i-1}$; a contradiction
 to $u$ and $v_{i-1}$ being distinct. In summary, $(v_{i-1},u,v_i,\mu)$ is
 distance-$(\ell-1)$-static. By Lemma \ref{lem:c4claim}, there is an induced
 $C_4$ containing the edges $\{v_{i-1}, v_{i}\}$ and $\{v_{i},u\}$ and an
 additional vertex $u'$ with $d_G(\mu,u')= \ell-2$ as well as the edges
 $\{v_{i-1},u'\}$ and $\{u,u'\}$. If $i=1$, then we found an induced $C_4$
 containing $\{v_{0}, v_{1}\} = \{x,y\}$; contradicting to the choice of $C$.
 Thus, $i\geq 2$ must hold. We claim that $u'\notin V_{i-2}$. To see this,
 assume, for contradiction, that $u' \in V_{i-2}$. Suppose first that $u'=v_0$
 and consider the cycle $C'$ traversing $u',v_1,\dots, v_{i},u,u'$ of length
 $|C'|=i+2$. If $|C|=i+2$, then $v_{i+2}=v_0$. Hence, $u'=v_0$ implies $v_{i+2}=
 u'$ and thus there is an edge $\{v_{i-1},u'\} = \{v_{i-1},v_{i+2}\} $ and $C$
 is not induced; a contradiction. Hence, $|C|>i+2$ must hold. But then $C'$ is
 shorter than $C$ and contains the edge $\{u',v_1\}=\{x,y\}$; a contradiction to
 the choice of $C$. Thus, $u'=v_j \in V_{i-2}\setminus\{ v_0\}$ must hold. Now
 consider the cycle $C''$ traversing $v_0,v_1,\dots,
 v_{j-1},u',u,v_{i+2},v_{i+3}\dots,v_{|C|-1},v_0$ ($v_1=u'$ or $v_1=v_{j-1}$ may
 possible). Note that $d_C(v_j,v_{i+2})\geq 4$ since $j\leq i-2$. Moreover, we
 have $d_G(v_j,v_{i+2})=2$ since there is a path with edges $\{u',u\} =
 \{v_j,u\},$ and $\{u, v_{i+2}\}$ in $G$. Hence, $C''$ is shorter than $C$ and
 contains $\{v_0,v_1\}=\{x,y\}$; again a contradiction. Hence, we found an
 induced $C_4$ with edges $\{v_{i-1}, v_{i}\}$, $\{v_{i},u\}$, $\{u,u'\}$,
 $\{u',v_{i-1}\}$ with $u\notin V_{i-1}$ and $u'\notin V_{i-2}$. Now, we can
 repeat the arguments on $u',v_{i-1},v_{i-2}$ and find an induced $C_4$ with
 edges $\{v_{i-2}, v_{i-1}\}$, $\{v_{i-1},u'\}$, $\{u',u''\}$, $\{u'',v_{i-2}\}$
 with $u'\notin V_{i-2},u''\notin V_{i-3}$. By induction and by repeating the
 latter arguments, one shows that all edges $\{v_{0}, v_{1}\} = \{x,y\}$,
 $\{v_{1}, v_{2}\}$, \dots, $\{v_{i}, v_{i+1}\}$ are located in some induced
 $C_4$; a contradiction to the choice $C$. Hence, any smallest cycle containing
 $\{x,y\}$ must be an induced $C_4$.
\end{proof}

\begin{lemma}\label{lem:2quartets->Q3MK23}
Let $G$ be a $k$-median graph with \mc vertex $\mu$. If $G$ contains two
distance-$\ell$-static quartets $(x,z,y,\mu), (x,z',y,\mu)\in V^4$, then $G$
contains an induced $K_{2,3}$ or $Q_3^-$. 
\end{lemma}
\begin{proof}
Let $\mu$ be a \mc vertex in $G=(V,E)$ and $(x,z,y,\mu), (x,z',y,\mu)\in V^4$ be
distance-$\ell$-static quartets. Hence, we have $d_G(x, \mu) = d_G(z, \mu) =
\ell = d_G(y, \mu) - 1$ and $d_G(x, z) = 2$ with $y$ a common neighbor of $x$
and $z$ as well as $d_G(x, \mu) = d_G(z', \mu) = \ell$ and $d_G(x, z') = 2$ with
$y$ a common neighbor of $x$ and $z'$. Thus, $y$ is a common neighbor of $z$
and $z'$. Since $G$ is bipartite, $d_G(z, z') = 2$. Moreover, $d_G(z,
\mu)=d_G(z', \mu) = \ell = d_G(y, \mu) - 1$. Thus, $(z,z',y,\mu)\in V^4$ is
distance-$\ell$-static. 

By Lemma \ref{lem:c4claim}, there is an induced $C_4$ in $G$ with edges
$\{x,y\}, \{y,z\}, \{z,u\}, \{u,x\}$ and an induced $C_4$ in $G$ with edges
$\{x,y\}, \{y,z'\}, \{z',u'\}, \{u',x\}$. If $u=u'$, then $G$ contains an
induced $K_{2,3}$ with bipartition $\{x,z\}\cupdot\{u,y,z'\}$. Hence, assume
that $u\neq u'$. Since $(z,z',y,\mu)\in V^4$ is distance-$\ell$-static, there is
an induced $C_4$ in $G$ with $\{z,y\}, \{y,z'\}, \{z',u''\}, \{u'',z\}$. Again
if $u''=u$ or $u''=u'$, there is an induced $K_{2,3}$. If $u,u',u''$ are
pairwise distinct, then $x,z,z',y,u,u',u''$ induce a $Q_3^-$.  
\end{proof}

\begin{proposition}
Let $G$ be a $Q_3^-$-free $k$-median graph with \mc vertex $\mu$. Then, $G$
contains an induced $K_{2,3}$ if and only if $G$ contains two
distance-$\ell$-static quartets $(x,z,y,\mu), (x,z',y,\mu)\in V^4$. 
\end{proposition}
\begin{proof}
The \emph{if} direction follows from Lemma \ref{lem:2quartets->Q3MK23} and the
fact that $G$ is $Q_3^-$-free. For the \emph{only-if} direction suppose that
$G=(V,E)$ contains an induced $K_{2,3}$ with bipartition
$\{x,x'\}\cupdot\{y,y',y''\}$. Let $v\in \{x,x',y,y',y''\}$ be one of the
vertices that is closest $\mu$ and suppose that $d_G(\mu,v)=\ell-1$ for some
$\ell\geq 1$. Assume, for contradiction, that $v=y$. Since $y$ is closest to
$\mu$ and $\{y,x\},\{y,x'\}\in E$, Lemma \ref{lem:adjdisttomedian} implies that
$d_G(\mu,x)=d_G(\mu,x')=\ell$. Since $\{y',x\},\{y',x'\}\in E$ and by Lemma
\ref{lem:adjdisttomedian}, we have $d_G(\mu,y') \in \{\ell-1, \ell+1\}$. In case
$d_G(\mu,y') = \ell-1$, on easily verifies that $y,y'\in I_G(\mu,x,x')$ which is
not possible, since $\med_G(\mu,x,x')$ is well-defined. But then, $d_G(\mu,y') =
\ell+1$ and $x,x'\in I_G(\mu,y,y')$ which yields the desired contradiction.
Therefore, $v\notin \{y,y',y''\}$ must hold. We may assume w.l.o.g.\ that $v=x$.
In this case, $x$ is closest to $\mu$ and $\{x,y\},\{x,y'\},\{x,y''\}\in E$
together with Lemma \ref{lem:adjdisttomedian} implies that
$d_G(\mu,y)=d_G(\mu,y')=d_G(\mu,y'') = \ell$. Again by Lemma
\ref{lem:adjdisttomedian} and since $x'$ is adjacent to $y,y'$ and $y''$, we
have $d_G(\mu,x') \in \{\ell-1, \ell+1\}$. If, however, $d_G(\mu,x') = \ell-1$,
then $x,x'\in I_G(\mu,y,y')$ which is not possible, since $\med_G(\mu,y,y')$ is
well-defined. Hence, we have $d_G(\mu,x') = \ell+1$. Thus, $G$ contains 
distance-$\ell$-static quartets $(y,y',x',\mu)$ and $(y,y'',x',\mu)$. 
\end{proof}

Median graphs are characterized as those graph in which the convex hull
of every isometric cycle is a hypercube. This property is, in general, 
not shared by $k$-median graphs. By way of example, the convex hull
of the isometric ``outer'' cycle $C_6$ of the graph $G$ in Fig.\ \ref{fig:exmpls}(c)
coincides with $G$ and contains an induced $Q_3^-$ and a $K_{2,3}$ but is, 
obviously, not a hypercube. However, the existence of isometric cycles on
six vertices in a $k$-median graph $G$ is a clear indicator for induced $Q_3^-$s
in $G$.

\begin{proposition}\label{prop:C6-Q3m}
A $k$-median graph $G$ contains an induced $Q_3^-$ if and only if $G$
contains an induced (or equivalently, an isometric) $C_6$. 	
\end{proposition}
\begin{proof}
It is an easy task to verify that the terms isometric and induced are equivalent 
for cycles on six vertices in case that $G$ is bipartite.
Moreover, one easily verifies  that any induced $Q_3^-\subseteq G$ contains an
induced $C_6$ (cf.\ Fig.\ \ref{fig:exmpls}(b)). Suppose now that $G=(V,E)$
contains an induced $C_6$, call it $C$.  
Let $\delta\colon V\to \mathbb{N}$
be the map that assigns to each vertex its distance to $\mu$, i.e.,
$\delta(v) = d_G(\mu,v)$. We assume that $V(C) = \{v_1,\dots, v_6\}$
and $E(C) = \{\{v_1,v_2\}, \dots \{v_5,v_6\}, \{v_6,v_1\}\}$. 
W.l.o.g.\ assume that $v_1$ is a vertex in  $C$ that is closest to 
$\mu$ and that $\delta(v_1)=\ell-1$ for some $\ell\geq 1$. 
When traversing $C$ from $v_1$ to $v_2$ \dots  to $v_6$
we obtain the ordered tuple $\Delta(C) \coloneqq (\delta(v_1), \delta(v_2) \dots, \delta(v_6))$. 
We denote by  $\Delta^{*}(C)$ the tuple $(\delta(v_1), \delta(v_6), \dots, \delta(v_2))$, 
i.e., the tuple obtained by traversing $C$ from $v_1$ to $v_6$ \dots  to $v_2$.

Full enumeration of all possible distances of the vertices in $C$
to $\mu$ and using Lemma  \ref{lem:adjdisttomedian}, shows that 
$\Delta(C)$ is always of one of the form  
\begin{center}
\begin{itemize}[noitemsep,nolistsep]
\item[] $\qquad$ $\Delta_1 \coloneqq (\ell-1, \ell, \ell-1,\ell, \ell-1, \ell)$; 
\item[] $\qquad$  $\Delta_2 \coloneqq  (\ell-1, \ell, \ell-1,\ell, \ell+1, \ell)$;
\item[] $\qquad$  $\Delta^*_2 \coloneqq (\ell-1, \ell, \ell+1,\ell, \ell-1, \ell)$;
\item[] $\qquad$  $\Delta_3 \coloneqq (\ell-1, \ell, \ell+1,\ell, \ell+1, \ell)$; or 
\item[] $\qquad$ $\Delta_4 \coloneqq (\ell-1, \ell, \ell+1,\ell+2, \ell+1, \ell)$.
\end{itemize}\end{center}
Suppose first that $\Delta(C) = \Delta_1 = (\ell-1, \ell, \ell-1,\ell, \ell-1, \ell)$. 
Then, $(v_3,v_5,v_4,\mu)$ is a distance-$(\ell-1)$-static quartet. 
By Lemma \ref{lem:c4claim}, there is an induced $C_4$ in $G$ with vertices $v_3,v_4,v_5,u$ and edges
$\{v_3,v_4\}, \{v_4,v_5\}, \{v_5,u\}, \{u,v_3\}$. In particular, $\delta(u)=\ell-2$. 
Since $C$ is induced, $u\notin V(C)$. If $u$ is adjacent to $v_1$, then $G$ contains an induced $Q_3^-$ and we are done. 
Hence, suppose that $\{u,v_1\}\notin E$. 
Now consider the distance-$(\ell-1)$-static quartet $(v_1,v_5,v_6,\mu)$.
Again there is an induced $C_4$ containing $v_1,v_5,v_6$ and $u'$ with  $\delta(u')=\ell-2$. 
If $u'=u$ or $u$ is adjacent to $v_3$, then $G$ contains an induced $Q_3^-$ and we are done. Hence, assume $u\neq u'$
and that $\{u',v_3\}\notin E$. 
Analogously, for the distance-$(\ell-1)$-static quartet $(v_1,v_3,v_2,\mu)$, 
Again there is an induced $C_4$ containing $v_1,v_2,v_3$ and $u''$ with  $\delta(u'')=\ell-2$.
If $u''=u$ or $u''=u'$ or $u''$ is adjacent to $v_5$, then $G$ contains an induced $Q_3^-$ and we are done.
Hence, we can assume that $u,u'$ and $u''$ are pairwise distinct
and that $\{u',v_5\}\notin E$.
Note that  $\delta(u)=\delta(u')=\delta(u'') = \ell-2$.
Moreover, by the latter arguments, the cycle traversing 
$v_1,u'',v_3,u,v_5,u'$ is an induced $C_6$, call it $C'$. 
Moreover, $\Delta(C') =  (\ell-1, \ell-2, \ell-1,\ell-2, \ell-1, \ell-2)$. 
Now we can repeat the latter arguments on $C'$ to obtain either an induced $Q_3^-$
or an induced $C_6$, call it $C''$, such that $\Delta(C'') =  (\ell-3, \ell-2, \ell-3,\ell-2, \ell-3, \ell-2)$. 
Since $G$ is finite, the latter process must terminate, i.e., in one of the steps
we cannot find a further induced cycle $\tilde C$ on six vertices, as the distances
in $\Delta(\tilde C)$ become closer and closer to $\mu$. Thus, $G$ must contain
an induced $Q_3^-$.  

Assume now that $\Delta(C) = \Delta_2 = (\ell-1, \ell, \ell-1,\ell, \ell+1, \ell)$. 
In this case, $(v_4,v_6,v_5,\mu)$ is a distance-$\ell$-static quartet. 
Hence, there is an induced $C_4$ in $G$ with vertices $v_4,v_5,v_6,u$. 
In particular, $\delta(u)=\ell-1$ and, since  $C$ is induced, $u\notin V(C)$. If $\{u,v_2\}\in E$, then 
$G$ contains a $Q_3^-$. Otherwise, if $\{u,v_2\}\notin E$, then $G$ contains 
an $C_6$, called $C'$, that is induced by $v_1,v_2,v_3,v_4,u,v_6$ and for which
$\Delta(C')=\Delta_1$. As already shown, this implies that $G$ contains a $Q_3^-$. 
The case $\Delta(C) = \Delta^*_2$ is shown analogously. 

Assume now that $\Delta(C) = \Delta_3 = (\ell-1, \ell, \ell+1,\ell, \ell+1, \ell)$
In this case, $(v_4,v_6,v_5,\mu)$ is a distance-$\ell$-static quartet
and there is an induced $C_4$ in $G$ with vertices $v_4,v_5,v_6,u$ 
and we have  $\delta(u)=\ell-1$. Since  $C$ is induced, $u\notin V(C)$.
If  $\{u,v_2\}\in E$, then $G$ contains a $Q_3^-$. Otherwise, the cycle $C'$
along the vertices $v_1,v_2,v_3,v_4,u,v_6$ is an induced $C_6$ and satisfies
$\Delta(C') = (\ell-1, \ell, \ell+1,\ell, \ell-1, \ell)$ and we can reuse
the arguments as in case $\Delta^*_2$ to conclude that $G$ contains an
induced $Q_3^-$. 

Finally suppose that $\Delta(C) = \Delta_4 =  (\ell-1, \ell, \ell+1,\ell+2, \ell+1, \ell)$. 
In this case, $(v_3,v_5,v_4,\mu)$ is a distance-$(\ell+1)$-static quartet
and there is an induced $C_4$ in $G$ with vertices $v_3,v_4,v_5,u$.
and we have  $\delta(u)=\ell$. Since  $C$ is induced, $u\notin V(C)$.
If  $\{u,v_1\}\in E$, then $G$ contains a $Q_3^-$. Otherwise, the cycle $C'$
along the vertices $v_1,v_2,v_3, u, v_5,v_6$ is an induced $C_6$ and satisfies
$\Delta(C') = (\ell-1, \ell, \ell+1,\ell, \ell+1, \ell) = \Delta_3$ 
and we can reuse the arguments as in case $\Delta_3$ to conclude that $G$ contains an
induced $Q_3^-$. 
\end{proof}

Partial cubes are graphs that have an isometric embedding into a hypercube. 
A cycle $C_6$ is a partial cube. Hence, Prop.\ \ref{prop:C6-Q3m} implies that
there are partial cubes that are not $k$-median graphs. Although every
induced $C_6$ in a $k$-median graph indicates the existence of a $Q_3^-$, a similar
result does not necessarily hold for larger isometric cycles, see Fig.\ \ref{fig:exmpl-cube}
for an example.

\begin{figure}[t]
\centering
\includegraphics[width=.6\textwidth]{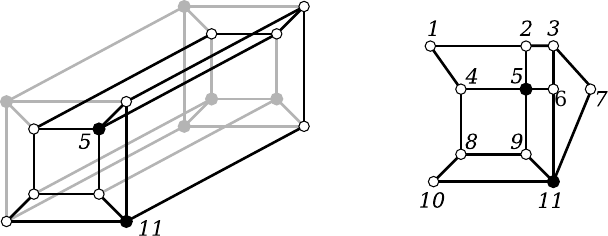}
\caption{A $K_{2,3}$-free proper 2-median graph $G$ (right). Its \mc vertices $5$ and $11$ are highlighted
			as black vertices. $G$ is an isometric subgraph of the 
			hypercube $Q_4$ and its isometric embedding is shown left. According to Prop.\ \ref{prop:C6-Q3m}, 
			each of its isometric cycles of length six
			(induced by $4,5,6,8,10,11$, resp., $2,3,5,7,9,11$) are contained in a $Q_3^-$ subgraph of $G$. 
			There is an isometric $C_8$ in $G$ induced by $1,2,3,4,7,8,10,11$ 
			whose convex hull is $G$. This $C_8$ is not contained in an induced $Q_4^-$ in $G$. }			
%			\TODO{Do induced $C_6$ always form the "outer-cycle" of a $Q_3^-$?}
			
\label{fig:exmpl-cube}
\end{figure}

Fig.\ \ref{fig:exmpls} shows that there are $k$-median graphs that contain induced 
$Q_3^-$s or $K_{2,3}$s. In fact, we have neither found examples nor clear arguments 
that shows that there are $K_{2,3}$- and $Q_3^-$-free $k$-median graphs
(except the median graphs whose isometric cycles are $C_4$s) or $K_{2,3}$-free $k$-median graphs that are not partial cubes. 
This, in particular, poses the following open problems.

	\begin{problem}\label{problem:Q3MK23}
	Does every $k$-median graph that is not a median graph contain
	an induced $Q_3^-$ or $K_{2,3}$? 	%
	\end{problem}		

	 Cor.\ \ref{cor:modK23med} shows that every $k$-median graph that is
	not a  median graph must contain an induced $K_{2,3}$. 	

	\begin{problem}\label{problem:hypercube}
	Is every $K_{2,3}$-free $k$-median graph an (isometric) subgraph of a median graph or a hypercube?
	\end{problem}

\section{Characterization of $\boldsymbol{k}$-median graphs}
\label{sec:convex}

\subsection{Convexity}

We start here with results utilizing a generalized notion of convexity.
To recall, a subgraph $H\subseteq G$ is $v$-convex, if $v\in V(H)$ and every
shortest path connecting $v$ and $x$ in $G$ is also contained in $H$, for all
$x\in V(H)$. Note that $v$-convex subgraphs on the same vertex sets do not
necessarily share the same edge set and thus may differ from each other. 
By way of example, in a $K_3$ a $v$-convex subgraph on three vertices could either be a path
$P_3$ or the graph $K_3$. See Figure \ref{fig:exmpl-v-convex} for further examples. Nevertheless, uniqueness for $v$-convex subgraphs
holds in case the underlying graph is bipartite.

\begin{lemma}\label{lem:v-convex-bip}
For every bipartite graph $G$, every $v$-convex subgraph is induced. Thus, for
all $W\subseteq V(G)$ and all $v\in W$, any two $v$-convex subgraphs of $G$ with
vertex set $W$ are identical.
\end{lemma}
\begin{proof}
Let $G$ be a bipartite graph and $H$ be a $v$-convex subgraph of $G$. Let
$\{x,y\}\in E(G)$ be an edge with $x,y\in V(H)$. By $v$-convexity of $H$, we
have $d_G(v,x)= d_{H}(v,x) = k$ and $d_G(v,y)= d_{H}(v,y) = k'$. 
Since $H\subseteq G$ is bipartite and connected, we can apply Lemma \ref{lem:01-diff-adge} to
conclude that $|k-k'|=1$ must hold. We may assume w.l.o.g.\ that $k'=k+1$.
Consider a shortest $(v,x)$-path $P$ in $G$. Since $\|P\| = k < d_G(v,y)$, the
vertex $y$ cannot be located on $P$. Thus, one can extend $P$ to a $(v,y)$-path
$P'$ in $G$ by adding the $\{x,y\}$. By construction, $\|P'\|=k+1$ and so $P'$
is a shortest $(v,y)$-path in $G$. Since $H$ is $v$-convex, we have $P'\subseteq
H$ and so $\{x,y\}\in E(H)$. Hence, $H$ is an induced subgraph of $G$. 
\end{proof}

Although $v$-convex subgraphs are not necessarily convex (cf.\ Fig.\
\ref{fig:exmpl-v-convex}), they are isometric subgraphs and $1$-median graphs
whenever $v$ is a \mc vertex in $G$.

\begin{lemma}\label{lem:rho-convex=>isometric}
Let $G$ be a $k$-median graph, $\mu\in V(G)$ be some \mc vertex in $G$ and $H$
be a $\mu$-convex subgraph of $G$. Then, $H$ is an isometric subgraph of $G$. In
particular, $H$ is a $1$-median graph with \mc vertex $\mu$ that satisfies
$\med_H(\mu,x,y) = \med_G(\mu,x,y)$ for all $x,y\in V(H)$.
\end{lemma}
\begin{proof}
Let $G$ be a $k$-median graph and $\mu\in V(G)$ be some \mc vertex in $G$ and
$H$ be a $\mu$-convex subgraph $H$ of $G$. We start with showing that $H$ is an
isometric subgraph of $G$, that is, $d_{H}(x,y) = d_{G}(x,y)$ for all $x,y\in
V(H)$. To this end, let $x,y\in V(H)$ be chosen arbitrarily. If $x=y$, the
statement is vacuously true. Hence, assume that $x\neq y$. Suppose first that
$x=\mu$. Since $H$ is a subgraph of $G$, we have $d_H(\mu,y)\not< d_G(\mu,y)$.
Moreover, $\mu$-convexity of $H$ implies that all shortest paths between $\mu$
and $y$ are also in $H$ and thus, $d_H(\mu,y) = d_G(\mu,y)$. Since $y$ was
chosen arbitrarily, we have $d_{H}(\mu,v) = d_{G}(\mu,v)$ for all $v\in V(H)$.

Assume now that $x$ and $y$ are both distinct from $\mu$. Since $\mu$ is a \mc
vertex in $G$, $z\coloneqq \med_G(\mu,x,y)$ is well-defined. We show first that
$d_H(z,x) = d_G(z,x)$. Clearly, if $x=z$, the statement is vacuously true. Thus,
assume that $x\neq z$. Since $d_{H}(\mu,v) = d_{G}(\mu,v)$ for all $v\in V(H)$,
it holds that $d_{H}(\mu,x) = d_{G}(\mu,x)$. This and $\mu$-convexity of $H$
implies that every shortest $(\mu,x)$-path in $G$ is also contained in $H$ and,
in particular, a shortest $(\mu,x)$-path in $H$. As $z$ is the unique median
between $\mu$, $x$ and $y$ in $G$, $z$ is located on some shortest
$(\mu,x)$-path $P$ in $G$ that is, by the latter arguments, also a shortest
$(\mu,x)$-path in $H$. This, in particular, implies that $z\in I_H(\mu,x)$.
Consider the subpath $P'$ of $P$ that connects $z$ and $x$. Since $P$ is a
shortest path in $H$, resp., $G$, we can conclude that $P'$ is shortest path in
$H$, resp., $G$ between $z$ and $x$. Hence, $d_{H}(z,x) = d_{G}(z,x)$. By
similar arguments, $z\in I_H(\mu,y)$ and $d_{H}(z,y) = d_{G}(z,y)$ must hold. By
the triangle inequality, $d_{H}(x,z) + d_{H}(z,y) \geq d_{H}(x,y)$. Since
$d_{H}(z,x) = d_{G}(z,x)$ and $d_{H}(z,y) = d_{G}(z,y)$, we can conclude that
$d_{G}(x,z) + d_{G}(z,y) = d_{H}(x,z) + d_{H}(z,y) $. Moreover, $z =
\med_G(\mu,x,y)$ and the fact that $H$ is a subgraph of $G$ implies that
$d_{H}(x,y) \geq d_{G}(x,y) =d_{G}(x,z) + d_{G}(z,y)$. Taking the latter three
arguments together yields \[d_{H}(x,y) \geq d_{G}(x,y) = d_{G}(x,z) + d_{G}(z,y)
= d_{H}(x,z) + d_{H}(z,y) \geq d_{H}(x,y).\] Hence, $d_{H}(x,y) = d_{G}(x,y)$ as
claimed and $z\in I_H(x,y)$. Therefore, $H$ is an isometric subgraph of $G$. In
particular, $z= \med_G(\mu,x,y)\in I_H(\mu,x)\cap I_H(\mu,y)\cap I_H(x,y) =
I_H(\mu,x,y)$ must hold.
 
It remains to show that $H$ is a $1$-median graph with \mc vertex $\mu$. By the
latter argument, $z \in I_H(\mu,x,y)$. We continue with showing that
$|I_H(\mu,x,y)| = 1$. Assume, for contradiction, that $|I_H(\mu,x,y)|\geq 2$ and
let $z'\in I_H(\mu,x,y)\setminus \{z\}$. Hence, $z'$ is located on some shortest
$(\mu,x)$-path $P$, some shortest $(\mu,y)$-path $P'$ and some shortest
$(x,y)$-path $P''$ in $H$. Since $H$ is isometric, we have $\|P\|= d_G(\mu,x)$,
$\|P'\|= d_G(\mu,y)$ and $\|P''\|= d_G(x,y)$. This and the fact that $H$ is a
subgraph of $G$ implies that $P$, $P'$ and $P''$ are shortest $(\mu,x)$-paths,
shortest $(\mu,y)$-paths and shortest $(x,y)$-paths in $G$. But then $z'\in
I_G(\mu,x)\cap I_G(\mu,y)\cap I_G(x,y)$ must hold which implies that
$\med_G(\mu,x,y)$ is not well-defined; a contradiction. In summary, $H$ is a
$1$-median graph with \mc vertex $\mu$. In particular, $\{z\} = I_H(\mu,x,y)$
holds and thus, $\med_H(\mu,x,y) = \med_G(\mu,x,y)$ for all $x,y\in V(H)$.
\end{proof}

Lemma \ref{lem:rho-convex=>isometric} holds, in general, only for 
$v$-convex subgraphs with $v$ being a \mc vertex. By way of example, 
the graph $G_2$ in Fig.\ \ref{fig:exmpl-v-convex} is a $k$-median graph
but the $v$-convex subgraphs for the non-\mc vertex is not isometric. 
Even more, the graph $G_4$ in Fig.\ \ref{fig:exmpl-v-convex} shows that 
$v$-convex subgraphs of $k$-median graphs are not necessarily $1$-median graphs.
Lemma \ref{lem:rho-convex=>isometric} can be used to obtain the following

\begin{proposition}\label{prop:rho-convex-root}
A vertex $\mu\in V(G)$ is a \mc vertex in a graph $G$ if and only if $\mu$ is a
\mc vertex in every $\mu$-convex subgraph of $G$.
\end{proposition}
\begin{proof}
The \emph{if} direction follows from the simple fact that $G$ is trivially a
$\mu$-convex subgraph of $G$. For the \emph{only-if} direction, assume that
$\mu\in V$ is a \mc vertex in $G$ and assume that $H$ is a $\mu$-convex subgraph
of $G$. By Lemma \ref{lem:rho-convex=>isometric}, $\mu$ is a \mc vertex in $H$.
\end{proof}

The definition of $k$-median graph and median graphs together with Prop.
\ref{prop:rho-convex-root} immediately implies
\begin{theorem}\label{thm:char-k-median}
$G=(V,E)$ is a $k$-median graph if and only if there are $k$ vertices
$\mu_1,\dots,\mu_k\in V$ such that $\mu_i$ is a \mc vertex in every
$\mu_i$-convex subgraphs of $G$, $1\leq i\leq k$. In particular, $G$ is a median
graph if and only if $v$ is a \mc vertex in every $v$-convex subgraph of $G$,
for all $v\in V(G)$.
\end{theorem}

\begin{corollary}\label{cor:I-rho-convex} 
Let $G$ be a $1$-median graph with $\mu$ being some \mc vertex in $G$. Then, for
all $v\in V(G)$, the subgraph $G[I_G(\mu,v)]$ induced by the interval
$I_G(\mu,v)$ is a $\mu$-convex, isometric subgraph of $G$ and a $1$-median graph
with \mc vertex $\mu$. 
\end{corollary}
\begin{proof}
	Let $G$ be a $1$-median graph with $\mu$ being some \mc vertex in $G$ and
	let $v\in V(G)$. Put $I\coloneqq I_G(\mu,v)$. By definition, all shortest
	paths in $G$ between $\mu$ and $v$ are contained in $G[I]$. Let $w\in I$. By
	Lemma \ref{lem:mulder-interval}, $I(\mu,w) \subseteq I$. Thus,
	all shortest $(\mu,w)$-paths are in $G[I]$ for all $w\in I$. Hence, $I$ is
	$\mu$-convex. The second statement now readily follows from Lemma
	\ref{lem:rho-convex=>isometric}
\end{proof}

Note that the converse of Cor.\ \ref{cor:I-rho-convex} is, in general, not
satisfied. To see this, consider the graph $G\simeq K_{2,3}$ with bipartition
$V(G) = X\cupdot Y$ where $X=\{x_1,x_2\}$ and $Y=\{y_1,y_2,y_3\}$. Put
$\mu\coloneqq y_1$. We have $I(\mu,y_1) = \{y_1\}$ and $I(\mu,y_i) =
\{\mu,x_1,x_2,y_i\}$ in case $i\in \{2,3\}$. In addition,
$I(\mu,x_i)=\{\mu,x_i\}$, $i\in \{1,2\}$. One easily observes that for each of
these intervals $I$, the induced subgraph $G[I]$ is $\mu$-convex, isometric and
a $1$-median graph with $\mc$ vertex $\mu$. However, since
$I(\mu,y_2,y_3)=\{x_1,x_2\}$, $\mu$ is not a \mc vertex of $G$.

The following result is well-known (cf.\ \cite{mulder1980interval}) and
the latter results allow us to establish a simple alternative proof. 

\begin{corollary}
$G$ is a median graph if and only if every convex subgraph of $G$ is a median
graph.
\end{corollary}
\begin{proof}
The \emph{if} direction follows from the simple fact that $G$ is a convex
subgraph of $G$. For the \emph{only-if} direction, suppose that $G$ is a median
graph and $H$ is some convex subgraph of $G$. Since $H$ is convex, $H$ is
$v$-convex for all $v\in V(H)$. By Theorem \ref{thm:char-k-median}, $v$ is a \mc
vertex in $H$ for all $v\in V(H)$. Hence, $H$ is a $|V(H)|$-median graph and,
therefore, a median graph. 
\end{proof}

It remains, however, an open question if a similar result holds for $k$-median
graphs as well.

\begin{problem}\label{problem:convex}
Is every convex  subgraph of a $k$-median graph $G$ a 1-median graph?
\end{problem}
An affirmative answer to Problem \ref{problem:convex} will
be provided in Prop.\ \ref{prop:convex-subgraph-modular} for modular  $k$-median graphs.

\subsection{Conditions \textit{(C0)}, \textit{(C1)} and \textit{(C2)}}

We provide now two further characterizations of $k$-median graphs. 
	To this end, we prove first
\begin{lemma}\label{lem:C1-dist}
	If a graph $G$ satisfies (C1) w.r.t.\ some vertex $u\in V(G)$, 
	then $d_G(u,a)\neq d_G(u,b)$ for all edges $\{a,b\}\in E(G)$
	and, thus, $G$ is bipartite.
\end{lemma}
\begin{proof}
	By contraposition, suppose that $u$ is a vertex in $G$
	for which there is an edge  $\{a,b\}\in E(G)$ such that 
	$d_G(u,a)= d_G(u,b)$. Hence, $I_G(a,b)=\{a,b\}$. 
	We have $b\notin I_G(u,a)$, since otherwise, 
	 $d_G(u,a)= d_G(u,b)+d_G(a,b)\neq d_G(u,b)$;
	  a contradiction. 
	  By similar arguments, 
	$a\notin I_G(u,b)$. Hence, $I_G(u,a,b) = \emptyset$
	and thus, $G$ does not satisfy (C1) w.r.t.\ $u$. 
	Lemma \ref{lem:01-diff-adge}, finally, implies that $G$ is bipartite.
\end{proof}

\begin{theorem}\label{thm:char3}
A graph $G=(V,E)$ is a $k$-median graph if and only if 
$G$ satisfies (C1) w.r.t.\ $\mu_1,\dots,\mu_k\in V$.
	In this case, the vertices $\mu_1,\dots,\mu_k\in V$ are \mc vertices of $G$.
\end{theorem}
\begin{proof}
	Let $G=(V,E)$ be a graph, $v,w\in V$ be chosen arbitrarily and put
	$I\coloneqq I_G(\mu,v,w)$. Suppose that $\mu$ is a \mc vertex in $G$. By
	Prop.\ \ref{prop:bipthm}, $G$ is bipartite. Moreover, since $\mu$ is a \mc
	vertex, we have $|I| = 1$. Thus, $G[I]\simeq K_1$ is non-empty and
	connected. Since $v,w$ where chosen arbitrarily, $G$ satisfies (C1) w.r.t.\
	$\mu$. 
	
	Assume now that $G$ satisfies (C1) w.r.t\ $\mu\in \{\mu_1,\dots,\mu_k\}$.
	Hence $G$ must be connected. Since $G[I]$ is not the empty graph, we have
	$I\neq \emptyset$. Assume, for contradiction, that $|I|>1$. Since $G$
	satisfies (C1) w.r.t\ $\mu$, $G[I]$ is connected. Thus, there are $a,b\in I$
	such that $\{a,b\}\in E(G)$. In the following, distances
	$d(\cdot,\cdot)=d_G(\cdot,\cdot)$ are taken w.r.t.\ $G$. By Lemma
	\ref{lem:C1-dist}, $d(\mu,a)\neq d(\mu,b)$ and Lemma \ref{lem:01-diff-adge}
	implies that $d(\mu,a)$ and $d(\mu,b)$ differ by exactly one, say $d(\mu,b)
	= d(\mu,a)+1$. Since $a,b\in I(\mu,v)$, it holds that $d(\mu,v) =
	d(\mu,a)+d(a,v) = d(\mu,b)+d(b,v)$. Thus,
	$d(b,v)=d(a,v)+d(\mu,a)-d(\mu,b)=d(a,v)-1$. By similar arguments and since,
	$a,b\in I_G(\mu,w)$, we obtain $d(b,w) = d(a,w)-1$. The latter two arguments
	together with $a,b\in I_G(v,w)$ imply that $d(v,w) = d(v,b)+d(b,w) =
	d(v,a)+d(a,w)-2 = d(v,w)-2$; a contradiction. Therefore, $|I|=1$ must hold.
	Since $v,w$ where chosen arbitrarily, $|I_G(\mu,v,w)|=1$ for all $v,w\in V$.
	Consequently, $\mu$ is a \mc vertex of $G$ and $G$ is a $k$-median graph.
\end{proof}

By  Theorem \ref{thm:mulder-median}, median graphs are interval-monotone and  
always satisfy (C0) w.r.t.\ all of its vertices. 
As argued in Section \ref{sec:basic-k-median}, $k$-median graphs are not necessarily interval-monotone. 
However, they always satisfy (C0) w.r.t.\ to its \mc vertices.
\begin{lemma}\label{lem:C0}
	$G$ satisfies (C0) w.r.t.\ $\mu$ for every \mc vertex $\mu\in V(G)$.
\end{lemma}
\begin{proof}
	Let $\mu$ be a \mc vertex of $G=(V,E)$.
	Suppose there are vertices 
	$v,w\in V$ such that $I(\mu,v)\cap I(v,w) = \{v\}$.
	 Since $\mu$ is a \mc vertex in $G$, it must hold that  $|I(\mu,v)\cap I(v,w)\cap I(\mu,w)|=1$
	and thus, $\med_G(\mu,v,w)=v$. Consequently, $v\in I(\mu,w)$ and $G$ satisfies (C0) w.r.t.\ $\mu$. 
\end{proof}
Theorem \ref{thm:mulder-median} and Lemma \ref{lem:C0}
beg the question to what extent Theorem \ref{thm:mulder-median}
can be generalized to cover the properties of $k$-median graphs. 
The next results provides an answer. 
	
\begin{theorem}\label{thm:char-C0}
A graph $G$ is a $k$-median graph if and only if 
$G$ is connected, bipartite and satisfies (C0) and (C2) w.r.t.\ $\mu_1,\dots,\mu_k\in V(G)$. 
In this case, the vertices $\mu_1,\dots,\mu_k\in V$ are \mc vertices of $G$.
\end{theorem}
\begin{proof}
	For the \emph{only-if} direction, assume first that $G=(V,E)$ is a
	$k$-median graph. By Prop.\ \ref{prop:bipthm}, $G$ is bipartite. Let $\mu\in
	V$ be one of its (at least $k$) \mc vertices. By Lemma \ref{lem:C0}, $G$
	satisfies (C0) w.r.t.\ $\mu$. Furthermore, since $\mu$ is a \mc vertex in
	$G$, we have $G[I(\mu,v,w)]\simeq K_1$ for all $u,v\in V$. Therefore, $G$
	trivially satisfies (C2) w.r.t. $\mu$.

	For the \emph{if} direction, assume now that $G$ is connected, bipartite and
	satisfies (C0) and (C2) w.r.t.\ $\mu_1,\dots,\mu_k$. Let $\mu\in
	\{\mu_1,\dots,\mu_k\}$. We show first that $I(\mu,v,w)\neq \emptyset$ for
	all $v,w\in V$. Since $G$ is connected, we can apply Lemma
	\ref{lem:mulder-interval} to conclude that there exists a vertex $z\in
	I(\mu,v)\cap I(v,w)$ such that $I(\mu,z)\cap I(z,w) = \{z\}$ for all $v,w\in
	V$. Since $G$ satisfies (C0) w.r.t.\ $\mu$, we have $z\in I(\mu,w)$ and so,
	$I(\mu,v,w)\neq \emptyset$ for all $v,w\in V$. Assume now, for
	contradiction, that $\mu$ is not a \mc vertex in $G$. Hence, there are
	vertices $v,w$ such that $\med_G(\mu,v,w)$ is not well-defined. By the
	latter arguments, $|I(\mu,v,w)|>1$. Since $G$ satisfies (C2) w.r.t.\ $\mu$,
	there is an edge $\{a,b\}\in E(G)$ with $a,b\in I(\mu,v,w)$. Bipartiteness
	of $G$ together with \ref{lem:01-diff-adge} implies that $d_G(\mu,a)$ and
	$d_G(\mu,b)$ differ by exactly one, say $d_G(\mu,b) = d_G(\mu,a)+1$. Now we
	can apply exactly the same arguments as in the proof of Theorem
	\ref{thm:char3} to conclude that $d_G(v,w) = d_G(v,w)-2$ and obtain the
	desired contradiction. Hence, $|I\mu,v,w)|=1$ for all $v,w\in V$.
	Consequently, $\mu$ is a \mc vertex of $G$ and $G$ is a $k$-median graph.
\end{proof}

The latter results have direct implications for median graphs.

\begin{theorem}\label{thm:med-novel}
	For every graph  $G$, the following statements are equivalent. 
	\begin{enumerate}
		\item $G$ is a median graph.
		\item $G$ satisfies (C1) w.r.t.\ all of its vertices. \label{eq:C1}
		\item $G$ is connected, bipartite and satisfies (C0)  and (C2) w.r.t.\ all of its vertices.
	\end{enumerate}
\end{theorem}
\begin{proof}
	The equivalence between (1) and (2) follows from Theorem \ref{thm:char3}
	and the fact that in median graphs all vertices are \mc vertices.
	Similarly, Theorem \ref{thm:char-C0} provides the equivalence between (1) and (3).
\end{proof}

We discussed the latter results, in particular Thm.\
\ref{thm:med-novel}\eqref{eq:C1}, with our friend and colleague Wilfried Imrich.
As he pointed out, the induced subgraphs $G[I_G(u,v,w)]$ in Condition (C1) can
be replaced by ``intersections of graphs'' to obtain an alternative
characterization of median graphs. To be more precise, let $G(u,v)$ be the
subgraph of $G$ with $V(G(u,v)) = I_G(u,v)$ and where $E(G(u,v))$ consists
precisely of all edges that lie on the shortest $(u,v)$-paths. Moreover, put
$G(u,v,w)\coloneqq G(u,v)\cap G(u,w)\cap G(v,w)$ for any $u,v,w\in V(G)$. On
easily verifies that $V(G(u,v,w))=I_G(u,v,w)$ and $G(u,v,w)\subseteq
G[I_G(u,v,w)]$. The difference between $G[I_G(u,v,w)]$ and $G(u,v,w)$ is
illustrated in Fig \ref{fig:exmpl-v-convex}. In this example, for the graph
$G''$, we have $I \coloneqq I_{G''}(a,b,c)=\{x,y\}$. The induced subgraph
$G''[I]\simeq K_2$ is connected and non-empty as it consists of the edge
$\{x,y\}$. In contrast, $G''(a,b,c)$ is disconnected and consists of the
vertices $x$ and $y$ only. Based on this idea, we obtain 
\begin{theorem}\label{thm:imrich}
 	$G$ is a median graph if and only if $G(u,v,w)$ is not empty and connected
 	for all $u,v,w\in V(G)$. 
\end{theorem}
\begin{proof}
	If 	$G=(V,E)$ is a median graph, then $V(G(u,v,w))=I_G(u,v,w)$ implies that 
	$|V(G(u,v,w))|=1$ and thus, 	$G(u,v,w)\simeq K_1$ for all $u,v,w\in V$. 
	Thus,  $G(u,v,w)$ is not empty and connected.
	
	Assume now that $G(u,v,w)$ is not empty and connected for all $u,v,w\in V(G)$.
	Since $G[I_G(u,v,w)]$ and $G(u,v,w)$ have the same 
	vertex sets and since $G(u,v,w)\subseteq G[I_G(u,v,w)]$
	it follows that $G$ satisfies (C1) w.r.t.\ all of its 
	vertices. Theorem \ref{thm:med-novel} implies that
	$G$ is a median graph. 
\end{proof}

\subsection{Modular Graphs}

For the sake of completeness, we provide here known results established by
Bandelt et al.\ for the special case of modular graphs and some of their
consequences.

\begin{theorem}[{\cite[Prop.\ 5.5]{BVV:93}}]\label{thm:charModMed}
A modular graph  $G$ is a $k$-median graph with \mc vertices $\mu_1,\dots,\mu_k$ if and only if 
the following statement is satisfied: 
If $u, v$ are vertices that have degree three in an induced subgraph $K_{2,3}\subseteq G$, then 
		$u \in I_G(\mu,v)$ or $v \in I_G(\mu,u)$ for all $\mu\in \{\mu_1,\dots,\mu_k\}$.
\end{theorem}
Since Theorem \ref{thm:charModMed} is always satisfied for $K_{2,3}$-free graphs, 
it implies

\begin{corollary}\label{cor:modK23->1med}
	Every modular $K_{2,3}$-free graph $G$ is a $k$-median graph for all $k\in
	\{1,\dots,|V(G)|\}$.
\end{corollary}

This and the fact that median graphs are modular and $K_{2,3}$-free implies

\begin{corollary}[{\cite[Thm.\ 3]{SM:99}}]\label{cor:med-modK23free}
	A graph $G$ is a median graph if and only if $G$ is modular and $K_{2,3}$-free. 
\end{corollary}

As a simple consequence of Cor.\ \ref{cor:med-modK23free}, we obtain the
following structural result which partially answer the question raised in
Problem \ref{problem:Q3MK23}.

\begin{corollary}\label{cor:modK23med}
	Every modular \mh{(}$k$-median\mh{)} graph that is not a median graph must
	contain an induced $K_{2,3}$.
\end{corollary}

Not all $k$-median graphs are modular. By way of example, a $Q_3^-$ is a
$4$-median graph but not modular. To see this, consider the three non-\mc
vertices $x,y,z$ in a $Q_3^-$ (cf.\ Fig.\ \ref{fig:exmpls}(b)). On easily
verifies that $I(x,y,z)=\emptyset$ must hold. 1-median graphs that are modular
are precisely the meshed graphs. To prove this, we first note that in
\cite{CCHO:20} weakly modular graphs have been defined as meshed graphs that
satisfy in addition a so-called triangle property. This triangle property is
trivially satisfied in bipartite graphs. Hence, in bipartite graphs the terms
meshed and weakly modular are equivalent. With this in hand, we can rephrase
Lemma 2.8 in \cite{CCHO:20} (see also \cite[Prop.\ 1.7]{BVV:93}) as 

\begin{lemma}
A graph $G$ is modular if and only if $G$ is connected, bipartite and meshed. 
\label{lem:charModular}
\end{lemma}
\begin{corollary}
 Let $G$ be a $k$-median graph. Then, $G$ is modular if and only if $G$ is meshed. 
\end{corollary}
\begin{proof}
If $G$ is modular, then Lemma \ref{lem:charModular} implies that $G$ is meshed.
Conversely, suppose that $G$ is meshed. Since $G$ is a $k$-median graph, it must
be connected. Moreover, Prop.\ \ref{prop:bipthm} implies that $G$ is bipartite.
By Lemma \ref{lem:charModular}, $G$ is modular. 
\end{proof}

We provide now a partial answer to the question raised in Problem \ref{problem:convex}.

\begin{proposition}\label{prop:convex-subgraph-modular}
	Every convex subgraph of a modular $k$-median graph is a $1$-median graph. 
\end{proposition}
\begin{proof}
	Let $G=(V,E)$ be a modular $k$-median graph, $\mu$ be a \mc vertex of $G$
	and $H$ be a convex subgraph of $G$. Assume first that $\mu\in V(H)$. Since
	$H$ is convex, $H$ is $\mu$-convex. By Prop.\ \ref{prop:rho-convex-root},
	$\mu$ is a \mc vertex in $H$. Thus, $H$ is a $1$-median graph. 
	
	Suppose now that $\mu\notin V(H)$ for any \mc vertex $\mu$ of $G$. Let $v\in
	V(H)$ be a vertex that is closest to $\mu$ among the vertices in $H$ and let
	$x,y\in V(H)$ be chosen arbitrarily. Since $G$ is modular, we have 
	$I_G(v,x,y)\neq \emptyset$. Let $P$ be a shortest $(v,x)$-path in
	$H$ (and thus, in $G$). Let $w\neq v$ be a vertex in $P$. If $d_G(\mu,w) +
	d_G(w,v) = d_G(\mu,v)$, then $w$ would be closer to $\mu$ than $v$. Hence,
	$d_G(\mu,w) + d_G(w,v) > d_G(\mu,v)$ must hold and, therefore, $w\notin
	I_G(\mu,v)$ for all $w\in V(P)\setminus \{v\}$ and every shortest
	$(v,x)$-path $P$. By similar arguments, none of the vertices on a shortest
	$(v,y)$-path can be contained in $I_G(\mu,v)$ except vertex $v$. 
	Hence,  $I_G(\mu,v)\cap I_G(v,x) = I_G(\mu,v)\cap I_G(v,y) = \{v\}$. 
	Since $G$ is a $k$-median graph,  Thm.\ \ref{thm:char-C0} implies that
	$G$ satisfies (C0) w.r.t.\ $\mu$. Taken the latter two arguments together
	shows that $v\in I_G(\mu,x)\cap I_G(\mu,y)$. 
	 Lemma \ref{lem:mulder-interval} implies that 
 	 $I_G(v,x)\subseteq I_G(\mu,x)$ and $I_G(v,y)\subseteq I_G(\mu,y)$.
 	 Hence, $I_G(v,x,y)\subseteq I_G(\mu,x,y)$.
	 Since $\mu$ is a \mc vertex, $I_G(\mu,x,y) = \{z\}$ for some $z\in V$.
	 The latter two arguments and the fact that $I_G(v,x,y)\neq \emptyset$
	 imply that $I_G(v,x,y)=\{z\}$. Since $z$ is located on a shortest
	 $(x,y)$-path in $G$ and since $H$ is convex it follows that $z\in V(H)$.
	 As the latter arguments hold
	 for all vertices $x,y\in V(H)$, that is, 
	 $\med_G(v,x,y) =  \med_H(v,x,y)$ for all $x,y\in V(H)$, it follows that 
	 $H$ is a 1-median graph with \mc vertex $v$. 
\end{proof}

Note that \mc vertices in a convex subgraph of a $k$-median graph $G$ are not
necessarily \mc vertices of $G$. By way of example, each induced $K_{2,3}$
subgraph of the graph $G_4$ in Fig.\ \ref{fig:exmpl-v-convex} is convex and, at
the same time, a 2-median graph. However, none of the \mc vertices within these
$K_{2,3}$s are \mc vertices of $G_4$.

\section{Summary and Outlook}

In this contribution, we considered $k$-median graphs as a natural
generalization of median graphs. We provided several characterizations of
$k$-median graphs based on a generalization of convexity ($v$-convexity) as well
as three simple conditions (C0),(C1) and (C2). These results, in turn, imply
several novel characterizations of median graphs in terms of the structure of
subgraphs based on three vertices and the respective shortest paths and
intervals between them. A simple tool written in python to verify if a given
graph is a $k$-median graph and to compute the largest such integer $k$ in the
affirmative case is provided at GitHub \cite{github-MH}.

	In the last decades, dozens of interesting characterizations of median
graphs have been established and we refer to \cite{SM:99} for an excellent
overview. It would be of interest to see in more detail how other
characterizations and results are linked to the structure of $k$-median graphs.
In particular, we want to understand in more detail how $K_{2,3}$-free
$k$-median graphs are linked to the structure of hypercubes and median graphs
(cf.\ Problem \ref{problem:hypercube}). Moreover, does every $k$-median graph
that is not a median graph contain an induced $K_{2,3}$ or $Q_3^-$ (Problem
\ref{problem:Q3MK23})? Are convex subgraphs of $k$-median graphs $1$ median
graphs (Problem \ref{problem:convex})?

	Furthermore, distance-$\ell$-static quartets $(x,z,y,\mu)$ in $k$-median
graphs determine $C_4$s induced by $x,z,y$ and an additional vertex $u$ that is
closer to the respective \mc vertex $\mu$. That is, repeated application of
distance-$\ell$-static quartets to reconstruct $C_4$s must terminate at a
certain point and provides valuable information about induced $C_4$s. This begs
the question to what extent a $k$-median graph can be reconstructed from its
distance-$\ell$-static quartets along large isometric cycles. Even more, which
type of $k$-median graphs can be uniquely determined by large isometric cycles
and the resulting distance-$\ell$-static quartets?

    The possibly most prominent characterization established by Martyn Mulder
\cite{Mulder:78,mulder1980interval} states that every median graph can be
obtained from a single vertex graph $K_1$ by a so-called convex expansion
procedure. While it is possibly a relative easy task to show that $k$-median
graphs are closed under the convex expansion procedure, it remains an open
question if every $k$-median graph can be obtained by convex expansions applied
to a particular set of starting graphs. Moreover, what is the connection of
$k$-median graphs to other generalizations or subclasses of median graphs, see
e.g.\
\cite{KLAVZAR2012462,BRESAR2007916,SEEMANN202338,BMW:94,BRESAR2003557,IMRICH1998677,BIK+02,BRESAR20071389}?
    
    So far, we have considered $k$-median graphs without making additional
assumptions on the integer $k$. For future research, it would be of interest to
investigate the structure of \emph{proper} $k$-median graphs for a fixed $k$.
Moreover, what are the requirements such that certain graph modification
operations (e.g.\ edge-deletion, edge-addition or contraction of edges) preserve
the property of a graph being a $k$-median graph?

\section*{Acknowledgments}
We want to thank Carmen Bruckmann, Peter F.\ Stadler, Wilfried Imrich and Sandi
Klav{\v{z}}ar for all the interesting discussions and useful comments to the
topic. The idea of Theorem \ref{thm:imrich} goes back to Wilfried Imrich. This
work was partially supported by the Data-driven Life Science (DDLS) program
funded by the Knut and Alice Wallenberg Foundation.

\bibliographystyle{spbasic}     
\bibliography{kmedian}

\end{document}